\documentclass{elsarticle}

   \renewcommand{\footnote}[1]{
  \textsuperscript{ 
    \addtocounter{footnote}{1} 
    (\thefootnote) 
  }
   \footnotetext{#1} 
}

\usepackage{bm}
\usepackage{amsmath}
\usepackage{amsthm}
\usepackage{trfsigns}
\usepackage{wasysym}
\usepackage{amssymb}
\usepackage{amsfonts}
\usepackage{graphicx}  
\usepackage[frenchb,english]{babel}
\usepackage[utf8]{inputenc}
\usepackage[T1]{fontenc}

\usepackage{mathrsfs}
\usepackage{pdfsync}
\usepackage{enumerate}
\usepackage{version}
\usepackage{calc}
\usepackage{subfigure}

 \usepackage{pstricks,pstricks-add,pst-math,pst-xkey}
 \usepackage{float}
 \usepackage{indentfirst}
 \usepackage{minitoc}
 \usepackage[title,titletoc]{appendix} 
 \usepackage{hyperref}
 \usepackage[mathscr]{eucal}

\DeclareMathAlphabet{\mathpzc}{OT1}{pzc}{m}{it}
\newtheorem{prop}{Proposition}

\newtheorem{thm}{Theorem}
\newtheorem{lem}{Lemma}
\theoremstyle{remark}
\newtheorem{remark}{Remark}
\newtheorem{notation}{Notation}
\newtheorem{definition}{Definition}

\def\v{\varepsilon}
\def\R{\mathbb{R}}

\def\N{\mathbb{N}}
\def\C{\mathbb{C}}

\def\O{\Omega}

\def\n{\nabla}
\def\p{\partial}
\def\*{u_{*}}

\def\x{\xi_{\varepsilon}}

\def\g{g_{\varepsilon}}
\def\be{\begin{eqnarray}}
\def\ee{\end{eqnarray}}

\def\1{\textrm{1\kern-0.25emI}}
\def\deg{\text{\rm deg}}
\def\dist{\text{\rm dist}}
\def\tr{\text{\rm tr}}

\def\H{\mathcal{H}}
\def\mint_#1{\mathchoice 
{\mathop{\vrule width 6pt height 3 pt depth -2.5pt 
\kern -8.8pt \intop}\nolimits_{#1}}%
{\mathop{\vrule width 5pt height 3 pt depth -2.6pt 
\kern -6.5pt \intop}\nolimits_{#1}}%
{\mathop{\vrule width 5pt height 3 pt depth -2.6pt 
\kern -6pt \intop}\nolimits_{#1}}%
{\mathop{\vrule width 5pt height 3 pt depth -2.6pt 
\kern -6pt \intop}\nolimits_{#1}}}  
\def\d{\displaystyle}

\def\diam{\text{\rm diam}}

\def\cl{\mathscr{C}}

\def\Eeps(u){E_\v(u):=\frac{1}{2}\int_A{|\n u|^2\di x}+\frac{1}{4\v^2}\int_A{\left(1-|u|^2\right)^2\di x}}

\def\di{\,\text{\rm d}}

\def\o{\omega}

\def\g{\gamma}
\def\S{\mathbb{S}}
\def\deucl{d_{\rm eucl}}

\def\Lip{{\rm Lip}}
\def\Card{{\rm Card}}
\def\Pos{{\mathscr{ P}}}
\def\Neg{\mathscr{N}}

\title{Study of a $3D$-Ginzburg-Landau functional with a discontinuous pinning term}
\author{Mickaël {\sc Dos Santos}\footnote{Université Paris-Est, LAMA–CNRS UMR 8050, 61, Avenue du Général de Gaulle, 94010 Créteil. France }\\\url{mickael.dos-santos@u-pec.fr}
}

\begin{document}

\begin{abstract}
In a convex domain $\O\subset\R^3$, we consider the minimization of a $3D$-Ginzburg-Landau type energy $E_\v(u)=\frac{1}{2}\int_\O|\n u|^2+\frac{1}{2\v^2}(a^2-|u|^2)^2$ with a discontinuous pinning term  $a$ among $H^1(\O,\C)$-maps subject to a Dirichlet boundary condition $g\in H^{1/2}(\p\O,\S^1)$. The pinning term $a:\R^3\to\R^*_+$ takes a constant value $b\in(0,1)$ in $\o$, an inner strictly convex subdomain of $\O$, and $1$ outside $\o$.
We prove energy estimates with various error terms depending on assumptions on $\O,\o$ and $g$. In some special cases, we identify the vorticity defects via the concentration of the energy. Under hypotheses on the  singularities of $g$ (the singularities are polarized and quantified by their degrees  which are $\pm 1$), vorticity defects are geodesics (computed w.r.t. a geodesic metric $d_{a^2}$ depending only on $a$) joining two paired singularities of $g$ $p_i\& n_{\sigma(i)}$ where $\sigma$ is a minimal connection (computed w.r.t. a metric $d_{a^2}$) of the singularities of $g$ and $p_1,...,p_k$ are the positive (resp. $n_1,...,n_k$ are the negative) singularities. 
\end{abstract}
\maketitle
\tableofcontents
\section{Introduction}
In a convex domain $\O\subset\R^3$, we consider the minimization of a $3D$-Ginzburg-Landau type energy with a discontinuous pinning term  among $H^1(\O,\C)$-maps subject to a Dirichlet boundary condition $g\in H^{1/2}(\p\O,\S^1)$. The pinning term $a:\R^3\to\R^*_+$ takes a constant value $b\in(0,1)$ in $\o$, an inner strictly convex subdomain of $\O$, and $1$ outside $\o$. The strict convexity of $\o$ is not necessary but it allows to make a simpler description of the techniques used in this article.

Our Ginzburg-Landau type energy is
\begin{equation}\label{5.GLpinning}
E_\v(u)=\frac{1}{2}\int_{\O}{\left\{|\n u(x)|^2+\frac{1}{2\v^2}\left[a(x)^2-|u(x)|^2\right]^2\right\}\,{\rm d}x}.
\end{equation}
In \eqref{5.GLpinning}, $u\in H^1_g:=\{u\in H^1(\O,\C)\,|\,\tr_{\p\O}u=g\}$.

We are interested in studying the {\it vorticity defects} of minimizers of $E_\v$ in $H^1_g$ via energetic estimates. In this article, letting $u_\v$ be such a minimizer, we aim in describing the set $\{|u_\v|\leq b/2\}$ (this set is the vorticity defects). In the asymptotic $\v\to0$ we expect that, at least for special $g$'s, the set $\{|u_\v|\leq b/2\}$ takes the form of a union of thin wires whose endpoints are in $\p\O$; under this form the vorticity defects are called {\it vorticity lines}. We also expect that a concentration of the energy occurs around this set.

Because the pinning term is discontinuous, an energetical noise appears in a small neighborhood of the discontinuity set of $a$ (this set is $\p\o$).

In order to study the minimization problem of $E_\v$ in $H^1_g$ we first consider an auxiliary minimization problem. Following \cite{LM1}, we let $U_\v$ be {\bf the} unique minimizer of $E_\v$ in $H^1_1:=\{u\in H^1(\O,\C)\,|\,\tr_{\p\O}u\equiv1\}$. The solution $U_\v$ plays an important role in the study. It allows to consider a decoupling of $E_\v$ (see Section \ref{S.Decupling-Lassoued-Mironescu}). If $v\in H^1(\O,\C)$ and $|v|\equiv1$ on $\p\O$, then \cite{LM1}
\[
E_\v(U_\v v)=E_\v(U_\v)+F_\v(v),\text{ where }F_\v(v)=\frac{1}{2}\int_\O\left\{U_\v^2|\n v|^2+\frac{U_\v^4}{2\v^2}(1-|v|^2)^2\right\}.
\]
Consequently the study of minimizers of $E_\v$ in $H^1_g$ is related to the study of minimizers of $F_\v$ in $H^1_g$. 

Our techniques are directly inspired from those initially developed by Sandier in \cite{Sandier1} (whose purpose was to give, in some special situations, a simple proof of the $3D$ analysis of the Ginzburg-Landau equation, by Lin and Rivière \cite{LR1}), and by their adaptations in \cite{BBM1}.

We prove energy estimates with various error terms depending on our assumptions on $\O$ and $g$ (see Theorems \ref{T5.Main1},\ref{T5.Main2} $\&$ \ref{T5.MainSymCase} in Section \ref{S5.MainStatements}). In some special cases, we identify the vorticity lines via the concentration of the energy. At the end of this section, we will present a strategy which could lead to the localization of the vorticity lines.

The results we present are a first step towards a more precise description of the vorticity defects and of the asymptotic of minimizers.

Before stating our own results, we start by recalling the asymptotic expansion of the energy in the standard $3D$-Ginzburg-Landau model (when $a\equiv1$).

For $g\in H^{1/2}(\p\O,\S^1)$, if we let
\[
E_\v^0(u)=\frac{1}{2}\int_\O\left\{|\n u|^2+\frac{1}{2\v^2}(1-|u|^2)^2\right\},
\]
then we have 
\begin{equation}\label{5.Form14}
\inf_{H^1_g}E_\v^0=C(g)|\ln\v|+o(|\ln\v|).
\end{equation}
Moreover, $\dfrac{C(g)}{\pi}$ is given by the length of a minimal connection connecting the singularities of $g$ (in the spirit of Brezis, Coron, Lieb \cite{BCL1}). (See \cite{LR1}, \cite{LR2}, \cite{Sandier1} and \cite{BBM1}).

For special $g$'s and for  a convex domain $\O$, \eqref{5.Form14} was obtained by Lin and Rivière \cite{LR1} (see also \cite{LR2}) and Sandier \cite{Sandier1}. The case of a general data $g\in H^{1/2}(\p\O,\S^1)$ and a simply connected  $\O$ is due to Bourgain, Brezis and Mironescu \cite{BBM1}.

The above articles are the main references in this work. One of our main results is the analog of \eqref{5.Form14} for the minimization of $F_\v$ (Theorem \ref{T5.Main1}). This result is first proved when $g$ is in a dense set $\mathcal{H}\subset H^{1/2}(\p\O,\S^1)$ and then extended by density. The upper bound is obtained directly using the techniques developed in \cite{Sandier1} and \cite{BBM1}. The lower bound needs an adaptation in the argument of Sandier \cite{Sandier1}. The main ingredient used to obtain a lower bound in \cite{Sandier1} is the existence of a "{\it structure function}" adapted to the singularities of $g$ (see Section \ref{S5.StructureFunction}). In the spirit of \cite{Sandier1}, we prove, under suitable assumptions on $\O$, $\o$ and $g$, the existence of structure functions adapted to our situation 
  (see Propositions \ref{C5.Ustruc}, \ref{C5.StructureFunctionCompact}, \ref{C5.StructureFunctionCompactBisBis} and \ref{P5.StructFonctSymCas}).

In our situation ($a=b\text{ in }\o\text{ and }a=1\text{ in }\R^3\setminus\o$), we have an analog of \eqref{5.Form14} for $\inf_{H^1_g}F_\v$ replacing $C(g)$ by $C(g,a)$. When $g$ admits a finite number of singularities, the constant ${{C(g,a)}}/{\pi}$ is the length of a minimal connection between the singularities of $g$ (see Section \ref{S5.PropertiesGeomObject} for precise definitions). This minimal connection is computed w.r.t. a metric $d_{a^2}$ depending only on $a$ (see \eqref{5.DefDistancea^2}). (This generalizes the case of the standard potential $(1-|u|^2)^2$, where the distance is the euclidean one.)

When $g$ has a finite number of singularities, one may prove a concentration of the energy along the vorticity lines (See Theorems \ref{T5.Main2} and \ref{T5.MainSymCase}).  As in \cite{LR1} and \cite{Sandier1}, we obtain, after normalization, that the energy of minimizer is uniform along the vorticity lines (See Theorem \ref{T5.Main2}). These vorticity lines are identified: they are geodesic segments associated to $d_{a^2}$.

In order to sum up our main results we state a theorem in a  simplified form. This theorem is a direct consequence of Theorems \ref{T5.Main1}, \ref{T5.Main2} and \ref{T5.MainSymCase} stated in Section \ref{S5.MainStatements}. 

\begin{thm}
Let $g\in H^{1/2}(\p\O,\S^1)$ then we have
\begin{enumerate}[$\bullet$]
\item $\inf_{H^1_g}E_\v=E_\v(U_\v)+C(g,a)|\ln\v|+o(|\ln\v|)$ where $E_\v(U_\v)\sim\v^{-1}$ and $C(g,a)$ depends only on the singularities of $g$ and on $a$ (it is the length of a minimal connection of the singularities of $g$ computed w.r.t. the distance $d_{a^2}$).
\item Let $g$ be a prepared boundary condition with a finite number of singularities of degree $\pm1$ ($g\in\H$, $\H$ defined in \eqref{TheNumberOfTheDenseSet}) and let $p_1,...,p_k$ (resp. $n_1,...,n_k$) be the positive (resp. negative) singularities of $g$. Let $\Gamma$ be the geodesic link of the singularities  (we assume that $\Gamma$ is unique), \emph{i.e.}, $\Gamma$ is a union of $k$ geodesic curves joining $p_i$ with $n_{\sigma(i)}$ where $\sigma$ is a permutation of $\{1,...,k\}$ s.t. the total $d_{a^2}$-length of the curves is minimal. 

Letting $v_\v$ be a minimizer of $F_\v$ in $H^1_g$ ({\it i.e.} $U_\v v_\v$ minimizes $E_\v$) we have
\[
\frac{\displaystyle\frac{U^2_\v}{2}|\n v_\v|^2+\frac{U^4_\v}{4\v^2}(1-|v_\v|^2)^2}{|\ln\v|}\,\mathscr{H}^3\text{ weakly converges in  $\O$ in the sense of the measures to }\pi a^2\mathscr{H}^1_{|\Gamma}.
\]   
Here $\mathscr{H}^3$ is the $3$-dimensional Hausdorff measure and $\mathscr{H}^1_{|\Gamma}$ is the one dimensional Hausdorff measure on $\Gamma$
\end{enumerate}
\end{thm}

The goal of this work is to explain how the vorticity lines are modified under the effect of a pinning term. 
Although from the theorems below we have an idea about the form of the vorticity defects, we do not identify exactly the set $\{|u_\v|\leq b/2\}$ (this set is the vorticity defects). In the study of Ginzburg-Landau type energies, it is standard to detect $\{|u_\v|\leq b/2\}$ in a first step via a concentration of the energy. Once this done, in a second step, coupling energy estimates with an {\it $\eta$-ellipticity result} (see below) we get that the set where we have a concentration of the energy corresponds to the vorticity defects.  In order to have a complete and rigorous description of the vorticity defects, we need  an $\eta$-ellipticity results in the spirit of \cite{BBOEllipticity} for the minimizers of $F_\v$. Namely:  fix $r>0$ then for small $\v$ and $v$ a minimizer of $F_\v$
\[
\text{ \it$\text{if, in a ball }B(x,r),\text{ the quantity }\frac{F_\v(v,B(x,r))}{|\ln\v|}\text{ is small, }\text{ then }|v(x)|\simeq1$.}
\]

It seems that an $\eta$-ellipticity result cannot be obtained by the standard method, which relies on a monotonicity formula obtained from a Pohozaev identity. The oscillating behavior of $U_\v$ yields impossible the direct application of monotonicity formulas. When $U_\v$ does not oscillate, it is possible to derive $\eta$-ellipticity (see \emph{e.g.} \cite{ZL1PinGL}). In our case, $\eta$-ellipticity would require a uniform control on the Lipschitz norm of $U_\v$; this does not hold in our situation.

Note that, the main result in \cite{BBOEllipticity} (an $\eta$-ellipticity result for critical points of $E_\v^0$) may be applied in a ball $B(x,r)$ s.t. $\overline{B(x,r)}\subset\O\setminus\p\o$. Indeed, let $u_\v=U_\v v_\v$ be a minimizer of $E_\v$ in $H^1_g$ and fix a ball $\overline{B(x,r)}\subset\O\setminus\p\o$. 

We let for $z\in B(0,a(x)r)$ $\tilde{u}_\v(z)={u_\v\left(x+\frac{z}{a(x)}\right)}/{a(x)}$ which solves $-\Delta u=\v^{-2}u(1-|u|^2)$ in $B(0,a(x)r)$. Note that if $x\notin\o$ then $\tilde{u}(z)=u(x+z)$. 

Because $\overline{B(x,r)}\subset\O\setminus\p\o$, we have $E_\v(U_\v,B(x,r))\to0$ and thus 
\[
\frac{F_\v(v_\v,B(x,r))}{|\ln\v|}\text{ is small }\Leftrightarrow\frac{E_\v(u_\v,B(x,r))}{|\ln\v|}\text{ is small }\Leftrightarrow\frac{E^0_\v(\tilde{u}_\v,B(0,a(x)r))}{|\ln\v|}\text{ is small}.
\]
We may apply the main result of \cite{BBOEllipticity} with $\tilde{u}_\v$ in $B(0,a(x)r)$ and, because $|\tilde{u}_\v(0)|=|u_\v(x)|/a(x)$, our results of concentration of the energy for $v_\v$ imply that outside the set of concentration of the energy {\bf and} "far from" $\p\o$ we have $|u_\v|> b/2$. 
 But this is not totally satisfying because, by a technical obstruction, we cannot cross over $\p\o$.


This paper is divided as follows: 
\begin{enumerate}[$\bullet$]
\item We first present the decoupling of Lassoued-Mironescu and fundamental properties of the special solution $U_\v$ (Section \ref{S5.PropertiesOfSpecial}).
\item In Section \ref{S5.PropertiesGeomObject}, we define and describe the main geometrical objects in the study (minimal connection/length, geodesic curve/link ...). Once this done, we state the main results in Section \ref{S5.MainStatements}.
\item The proofs of the main results are sketched in Section \ref{S5.OutlineProofs}. In particular we explain how we may use \cite{Sandier1} and  \cite{BBM1}  and we underline  the required adaptations.
\item The heart of the argument is based on energetic estimates. In Section \ref{S5.thetwoUpperBounds} we proof upper bounds according to various assumptions (which give various error terms) and in Section \ref{S5.LowerBoundSandier} we obtain lower bounds. Section \ref{S5.StructureFunction} is  dedicated to the key tool used in Section \ref{S5.LowerBoundSandier}.
\item Section \ref{S5.DensityArgument} is devoted to the last argument in the proof of Theorem \ref{T5.Main1}.
\end{enumerate}

 \section{The decoupling of Lassoued-Mironescu}\label{S.Decupling-Lassoued-Mironescu}
 \label{S5.PropertiesOfSpecial}
Let $\overline{\omega}\subset\O\subset\R^3$ be two smooth bounded open sets s.t. $\O$ is convex and $\o$ is strictly convex. For $b\in(0,1)$ we define 
\begin{equation}\label{5Pinnineddef}
\begin{array}{cccc}
a:&\R^3&\to&\{b,1\}
\\
&x&\mapsto&\begin{cases}b&\text{if }x\in\o\\1&\text{otherwise}\end{cases}
\end{array}.
\end{equation}
We denote by $E_\v$ the Ginzburg-Landau functional with $a$ as pinning term, namely
\begin{equation}\nonumber
E_\v(u)=\frac{1}{2}\int_{\O}{\left\{|\n u(x)|^2+\frac{1}{2\v^2}\left[a(x)^2-|u(x)|^2\right]^2\right\}\,{\rm d}x}.
\end{equation}

For $\v>0$, we let (see \cite{LM1}) $U_\v$ be {\bf the} unique global minimizer of $E_\v$ in 
\[
H^1_1:=\{u\in\ H^1(\O,\C)\,|\,\tr_{\p\O}u\equiv1\}.
\] 

In the following, we will denote also $U_\v\in H_{\rm loc}^1(\R^3,\C)$ the extension by $1$ of the unique global minimizer of $E_\v$ in $H^1_1$.

\begin{prop}\label{P5.Fond1}
Let $U_\v$ be a minimizer of $E_\v$ in $H^1_1$, then $U_\v$ is unique and the following assertions are true 
\begin{enumerate}[1.]
\item $U_\v:\R^3\to[b,1]$ (from \cite{LM1}),
\item $\d-\Delta U_\v=\frac{1}{\v^2}U_\v(a^2-U_\v^2)$ in $\O$,
\item $\d E_\v(U_\v)\underset{\v\to0}{\sim}\frac{1}{\v^2}\int_\O{(a^2-U_\v^2)^2}\underset{\v\to0}{\sim}\frac{1}{\v}$ (same argument as in  \cite{LM1}),
\item There are $C,\gamma>0$ 
s.t. for $x\in\O$ we have (same proof as in \cite{Publi3} Proposition 2)
\begin{equation}\label{5.Uepsprochea}
|U_\v(x)-a(x)|\leq C\e^{-\gamma \dist(x,\p\o)/\v},
\end{equation}
\item If $v\in  H^1(\O,\C)$ is s.t. $|\tr_{\p\O}v|=1$ then $E_\v(U_\v v)=E_\v(U_\v)+F_\v(v)$ (same proof as \cite{LM1}) with
\begin{equation}\label{5.GLpoids}
F_\v(v)=\frac{1}{2}\int_\O{\left\{U_\v^2|\n v|^2+\frac{U_\v^4}{2\v^2}(1-|v|^2)^2\right\}},
\end{equation} 

\item If $v$ minimizes $F_\v$ in $H^1_g:=\{v\in H^1(\O,\C)\,|\,\tr_{\p\O}v=g\}$, then $|v|\leq1$ in $\O$ (same proof as \cite{LM1}).
\end{enumerate}
\end{prop}
Assertion 5 of Proposition \ref{P5.Fond1} expresses that minimization of $E_\v$ is related with those of $F_\v$. This reformulation is standard in the context of pinned Ginzburg-Landau type energies. It allows to make a cleaning  of the energy in order to detect concentration of energy.

To understand this cleaning and our main results (Theorems \ref{T5.Main1}$-$\ref{T5.MainSymCase}) we may see the minimization of $E_\v$ w.r.t. the boundary condition $g\equiv1$ as the cheapest minimization problem in terms of the energy. Indeed, for $u\in H^1(\O,\C)$ s.t. $|\tr_{\p\O}u|\equiv1$, we have $E_\v(u)\geq E_\v(|u|)\geq E_\v(U_\v)\underset{\v\to0}{\sim}\frac{1}{\v}$. Note that with the special boundary condition $g\equiv1$ and in the study of a simplified Ginzburg-Landau energy without pinning term $E_\v^0$ (obtained from $E_\v$ by taking $a\equiv1$) we have that $U_\v^0$, {\bf the} global minimizer of $E_\v^0$ in $H^1_1$, is constant equal to $1$; and therefore $E_\v^0(U_\v^0)=0$. In the presence of a pinning term $a$ (given by \eqref{5Pinnineddef}), the special solution $U_\v$ carries a concentration of energy along the singular surface $\p\o$.  More precisely we may prove a refined version of Proposition \ref{P5.Fond1}.3: for $\eta>0$ we have 
\[
E_\v(U_\v)\underset{\v\to0}{\sim} E_\v(U_\v,\{x\in\O\,|\,\dist(x,\p\o)<\eta\})\underset{\v\to0}{\sim}\frac{1}{\v}
\]
and $E_\v(U_\v,\{x\in\O\,|\,\dist(x,\p\o)>\eta\})\to0$ when $\v\to0$. Here for $V$ an open subset of $\O$ we denoted $E_\v(U_\v,V)=\frac{1}{2}\int_V\left\{|\n U_\v|^2+\frac{1}{2\v^2}(a^2-U_\v^2)^2\right\}$.

One of the advantage of the Lassoued-Mironescu's decoupling is that, in the study of $E_\v(\cdot)=F_\v(\frac{\cdot}{U_\v})+E_\v(U_\v)$, the singular part of $E_\v$ is entirely carried by the fixed number $E_\v(U_\v)$. 

Our best results express this fact: for $g\in H^{1/2}(\p\O,\S^1)$ 
\begin{enumerate}[$\bullet$]
\item The minimal energy takes a standard form: 
\[
\inf_{H^1_g}E_\v=E_\v(U_\v)+\inf_{H^1_g}F_\v=E_\v(U_\v)+C(g,a)|\ln\v|+o(|\ln\v|)
\]
(see Theorems \ref{T5.Main1}$\&$\ref{T5.MainSymCase}) where $E_\v(U_\v)\sim\v^{-1}$;
\item Under some hypotheses the concentration of the energy of $F_\v$ is uniform along curves (see Theorem \ref{T5.Main2}). These curves are called \emph{vorticity lines}.
\end{enumerate}
To end this section we may underline the interpretation of Lassoued-Mironescu's decoupling $E_\v(\cdot)=F_\v(\frac{\cdot}{U_\v})+E_\v(U_\v)$ in terms of concentration of energy:
\begin{enumerate}[$\bullet$]
\item At the first order ($\v^{-1}$), the interface $\p\o$ corresponds to an "artificial" singular surface whose energetic cost is contained in $E_\v(U_\v)$,
\item At the second order ($|\ln\v|$), minimizers for $F_\v(\cdot/U_\v)$ concentrate their energy along the vorticity lines.  These curves are described in Section \ref{S5.DescribtionSectionGeode}, they can be interpreted as "line segments" bent by the interface $\p\o$.
\end{enumerate}

\section{Minimal connections, geodesic links}\label{S5.PropertiesGeomObject}
In this section we define the main geometrical objects needed in the description of the vorticity lines.

In subsection \ref{S5.Distribution} we present a standard method which described the singularities of a map $g\in H^{1/2}(\p\O,\S^1)$: to a map $g\in H^{1/2}(\p\O,\S^1)$ we associate a distribution $T_g$. In Proposition \ref{P5.BBM1AuxProp} we present the main properties of $T_g$ used in this article. We define also an important dense set of $H^{1/2}(\p\O,\S^1)$.

In subsection \ref{S5.DescribtionSectionGeode} we state some definitions and we describe the metric $d_{a^2}$ which plays an important role in this study. 
\subsection{Length of a minimal connection of a map $g\in H^{1/2}(\p\O,\S^1)$}\label{S5.Distribution}
For $g\in H^{1/2}(\p\O,\S^1)$, following \cite{BBM1}, one may associate to $g$ a continuous linear form 
\[
T_g:({\rm Lip}(\p\O,\R),{\|\cdot\|_{\rm Lip}})\to\R.
\]
Here  ${\rm Lip}(\p\O,\R)$ is the space of Lipschitz functions equipped with the standard norm $\|\varphi\|_{\rm Lip}=\|\varphi\|_{L^\infty}+\displaystyle\sup_{\substack{x,y\in\p\O\\x\neq y}}\frac{|\varphi(x)-\varphi(y)|}{|x-y|}$ with $|x-y|=\deucl(x,y)$ is the euclidean distance in $\R^3$ between $x$ and $y$. 


For $g\in H^{1/2}(\p\O,\S^1)$, the map $T_g:{\rm Lip}(\p\O,\R)\to\R$ is defined by the following way: let $\varphi\in\Lip(\p\O,\R)$
\begin{enumerate}[$\bullet$]
\item fix $u\in H^1_g$ and consider $H=2(\p_2 u\wedge\p_3 u\,,\,\p_3 u\wedge\p_1 u\,,\,\p_1 u\wedge\p_2 u)$;
\item  fix $\phi\in{\rm Lip}(\O,\R)$ s.t. $\phi=\varphi$ on $\p\O$;
\end{enumerate}
then $\displaystyle\int_{\O}{H\cdot\n\phi}$ is independent of the choice of $u$ and $\phi$.

Therefore we may define the continuous linear form \[
\begin{array}{cccc}
T_g:&{\rm Lip}(\p\O,\R)&\to&\R\\&\varphi&\mapsto&\displaystyle\int_{\O}{H\cdot\n\phi}
\end{array}.
\]

\begin{notation}
Here "$\wedge$" stands for the "vectorial product" in $\C$: $(x_1+\imath y_1)\wedge(x_2+\imath y_2)=x_1y_2-x_2y_1$, $x_1,x_2,y_1,y_2\in\R$.
\end{notation}
Following \cite{BBM1}, we denote, for $g\in H^{1/2}(\p\O,\S^1)$ and $d$ an equivalent distance with $\deucl$ on $\p\O$, 
\begin{equation}\label{5.DeflongMin1}
L(g,d):=\frac{1}{2\pi}\sup\left\{T_g(\varphi)\,\left|\,|\varphi|_{d}\leq1\right.\right\}=\frac{1}{2\pi}\max\left\{T_g(\varphi)\,\left|\,|\varphi|_{d}\leq1\right.\right\}
\end{equation}
with 
\[
|\varphi|_{d}:=\sup_{\substack{x\neq y\\x,y\in\p\O}}\frac{|\varphi(x)-\varphi(y)|}{d(x,y)}.
\]
Note that, since $T_g:({\rm Lip}(\p\O,\R),{\|\cdot\|_{\rm Lip}})\to\R$ is continuous and $d,\deucl$ are equivalent on $\p\O$, then  $L(g,d)$ is finite.

In the spirit of \cite{LR1},\cite{Sandier1} and \cite{BBM1} we deal with prepared boundary conditions $g$'s. In this article we use the dense subset  $\mathcal{H}\subset H^{1/2}(\p\O,\S^1)$
\begin{equation}\label{TheNumberOfTheDenseSet}
\mathcal{H}=\left\{g\in \bigcap_{1\leq p<2}{W^{1,p}(\p\O,\S^1)}\,\left|\,\begin{array}{c}\text{$g$ is smooth outside a finite set $\cl$,  }\\\forall M \in\cl\text{ we have for $x$ close to $ M $:}\\|\n g(x)|\leq C/|x- M |,\\\exists\,R_ M \in\mathcal{O}(3)\text{ s.t. }\left|g(x)-R_ M \left(\frac{x- M }{|x- M |}\right)\right|\leq C|x- M |.\end{array}\right.\right\}.
\end{equation}
Here we considered $\S^1\simeq\S^1\times\{0\}\subset \S^2$. 

One may define $\deg(u, M)$, the topological degree of $u$ with respect to $M$: if $ R_ M \in\mathcal{O}(3)^+$ then $\deg(u, M)=1$ otherwise $\deg(u, M)=-1$.

In order to justify the term of "degree", assume that in a neighborhood of $M\in\cl$, $\p\O$ is flat. Then, for $r>0$ sufficiently small, $C=\p B(M,r)\cap\p\O$ is a circle centered in $M$. This circle has a natural orientation induced by $B(M,r)\cap\O$. Thus $g_{|C}\in C^\infty(C,\S^1)$ admits a well defined topological degree (see \emph{e.g.} \cite{B1}), and this degree does not depend on small $r$.

We may partition the set $\cl$ into two sets: the positive singularities and the negative singularities. We consider
\[
\Pos=\{ M \in\cl\,|\,\deg(u, M )=1\}\text{ and }\Neg=\{ M \in\cl\,|\,\deg(u, M )=-1\}.
\]
One may also consider for $g\in\mathcal{H}$ the degree of $g$ with respect to $\p U$ for $U$ a non empty smooth open set of $\p\O$ s.t. $\p U$ does not contain any singularities of $g$. This degree is defined as 
\begin{equation}\label{5.DefDegOuvert}
\deg(g,\p U)=\Card(\{p\in \Pos\,|\,p\in U\})-\Card(\{n\in \Neg\,|\,n\in U\}).
\end{equation}

From \cite{BBM1}, we have the following
\begin{prop}\label{P5.BBM1AuxProp}
Let $g,h\in H^{1/2}(\p\O,\S^1)$, then we have
\begin{enumerate}[1.]
\item $T_{gh}=T_g+T_h$ and $T_{\overline{g}}=-T_g$ (Lemma 9 \cite{BBM1}),
\item $|(T_g-T_h)(\varphi)|\leq C|g-h|_{H^{1/2}}(|g|_{H^{1/2}}+|h|_{H^{1/2}})|\varphi|_{\deucl}$, $\varphi\in{\rm Lip}(\p\O,\R)$ (Lemma 9 \cite{BBM1}),
\item $\mathcal{H}$ is dense in $H^{1/2}(\p\O,\S^1)$ (Lemma B.1 \cite{BBM1}),
\item if $u\in\mathcal{H}$, then $\Card(\Pos)=\Card(\Neg)$ and $T_g=\displaystyle2\pi\sum_{p\in \Pos}\delta_{p}-2\pi\sum_{n\in \Neg}\delta_{n}$ (Lemma 2 \cite{BBM1}),
\item if $u\in\mathcal{H}$, then $L(g,d)=\min_{\sigma\in S_k}\sum_i{d(p_i,n_{\sigma(i)})}$ where $d$ is a distance equivalent with $\deucl$ on $\p\O$ (Theorem 1 \cite{BBM1}). Here $S_k$ is the set of the permutations of $\{1,...,k\}$.
\end{enumerate}
\end{prop}
\subsection{Minimal connections, minimal length, geodesic links}\label{S5.DescribtionSectionGeode}
In the last assertion of Proposition \ref{P5.BBM1AuxProp}, we used the notion of length of a minimal connection w.r.t. a distance $d$.

Namely, consider $d$ a distance on $\cl=\Pos\cup \Neg$, $\Pos, \Neg\subset\R^3$ two sets of $k$  distinct points s.t. $\Pos\cap \Neg=\emptyset$, $\Pos=\{p_1,...,p_k\}$ and $\Neg=\{n_1,...,n_k\}$. 

\begin{definition}We denote by $L(\cl,d)$ the {\it length of a minimal connection} of $\cl$ in $(\cl,d)$, \emph{i.e.}, 
\begin{equation}\label{5.longcon}
L(\cl,d)=\min_{\sigma\in S_k}{\sum_{i=1}^k{d(p_i,n_{\sigma(i)})}}.
\end{equation}
\end{definition}
In \cite{BCL1} (Lemma 4.2), the authors proved that
\begin{equation}
L(\cl,d)=\max\left\{\left.\sum_{i=1}^k\left\{\varphi(p_i)-\varphi(n_i)\right\}\,\right|\,\varphi:\cl\to\R,\,|\varphi|^\cl_d\leq1\right\}
\end{equation}
with
\[
|\varphi|^\cl_{d}=\sup_{\substack{x\neq y\\x,y\in\cl}}\frac{|\varphi(x)-\varphi(y)|}{d(x,y)}.
\]

\begin{definition}A permutation $\sigma$ s.t. $\sum_i{d(p_i,n_{\sigma(i)})}=L(\Pos\cup \Neg,d)$ is called a {\it minimal connection} of $(\Pos\cup \Neg,d)$.
\end{definition}
In the following we will consider a special form of distance $d$ on $\p\O$: a geodesic distance in $\overline{\O}$ equipped with a metric that we will describe below.

Let us first introduce some notations. Let $f:\R^3\to[b^2,1]$ be a Borel function and let $\Gamma\subset\overline{\O}$ be  Lipschitz curve. We denote by ${\rm long}_f(\Gamma)$ the length of $\Gamma$ in the metric $fh$ (here $h$ is the euclidean metric in $\R^3$), \emph{i.e.}, 
\[
{\rm long}_f(\Gamma):=\int_0^1f\left[\gamma(s)\right]\,|\gamma'(s)|{\rm d}s,\,\gamma:[0,1]\to\Gamma\text{ is an admissible parametrization of  $\Gamma$}.
\]
In this paper, when we consider a curve (or arc) $\Gamma$, it will be implicitly that it is a Lipschitz one.

We define $d_f$ as the geodesic distance in $f h$ ($h$ is the euclidean metric in $\R^3$).

Thus, for $x,y\in\R^3$, $x\neq y$, we have
\begin{equation}\label{5.DefDistancea^2}
d_f(x,y)=\inf_{\substack{\Gamma\text{ Lipschitz arc }\\\text{with endpoints $x,y$}}}{\rm long}_f(\Gamma).
\end{equation}


In the special case $f=a^2$, one may easily prove the following proposition
\begin{prop}\label{P5.ClassificationGeodesiquespecialcase}
Let $x,y\in\R^3$, $x\neq y$. The following assertions are true
\begin{enumerate}[1.]
\item In \eqref{5.DefDistancea^2} the \emph{infimum} is attained. 
\item[]We denote by $\Gamma_0 $ a minimal curve in \eqref{5.DefDistancea^2}.
\item If $x,y\in\overline{\O}$ then a geodesic $\Gamma_0$ is included in $\overline{\O}$.
\item A geodesic $\Gamma_0$ is a union of at most three line segments: $\Gamma_0 =\cup_{i=1}^l S_i$
\item These line segments are such that
\begin{enumerate}[a.]
\item if $x,y\in\overline{\o}$ then $l=1$,
\item if $l=2$ then $S_1\cap S_2\subset\p\o$,
\item if $l=3$ then $x,y\in\R^3\setminus\overline{\o}$ and $S_2$ is a chord of $\o$,
\item if $[x,y]\cap\overline{\o}=\{z\}$ then $l\in\{2,3\}$.
\end{enumerate}
\item For $M\in\R^3$ we have $\psi_M=d_{a^2}(\cdot,M)\in W_{\rm loc}^{1,\infty}(\R^3,\R)$ and $|\n \psi_M|=a^2$.
\end{enumerate}
\end{prop}

\begin{definition}
In the case $d=d_{a^2}$ and $\cl=\Pos\cup \Neg\subset\p\O$, we say that $\cup_i{\Gamma_i}$ is a {\it geodesic link} when $\sigma$ is a minimal connection in  $(\cl,d_{a^2})$ and $\Gamma_i$ is a geodesic joining  $p_i$ to $n_{\sigma(i)}$.
\end{definition}
 In Figure \ref{F5.GeodesicLink}, we have represented a geodesic link for $k=2$ and a certain $b\in(0,1)$. 
\begin{remark}\label{R5.GeodLink}We consider the situation $d=d_{a^2}$ and $\cl=\Pos\cup \Neg\subset\p\O$.
\begin{enumerate}
\item If the minimal connection is unique, then, letting $\cup_i{\Gamma_i}$ be a geodesic link, we have $\Gamma_i\cap\Gamma_j=\emptyset$ for $i\neq j$. In particular, if the geodesic link $\cup_i{\Gamma_i}$ is unique then  $\Gamma_i\cap\Gamma_j=\emptyset$ for $i\neq j$.
\item If a geodesic link $\cup_i{\Gamma_i}$ is s.t. for all $i$ $\Gamma_i$ is a line segment, then we have $\Gamma_i\cap\Gamma_j=\emptyset$ for $i\neq j$.
\end{enumerate}
\end{remark}
\begin{figure}[h]
\begin{center}
        \includegraphics[width=15cm]{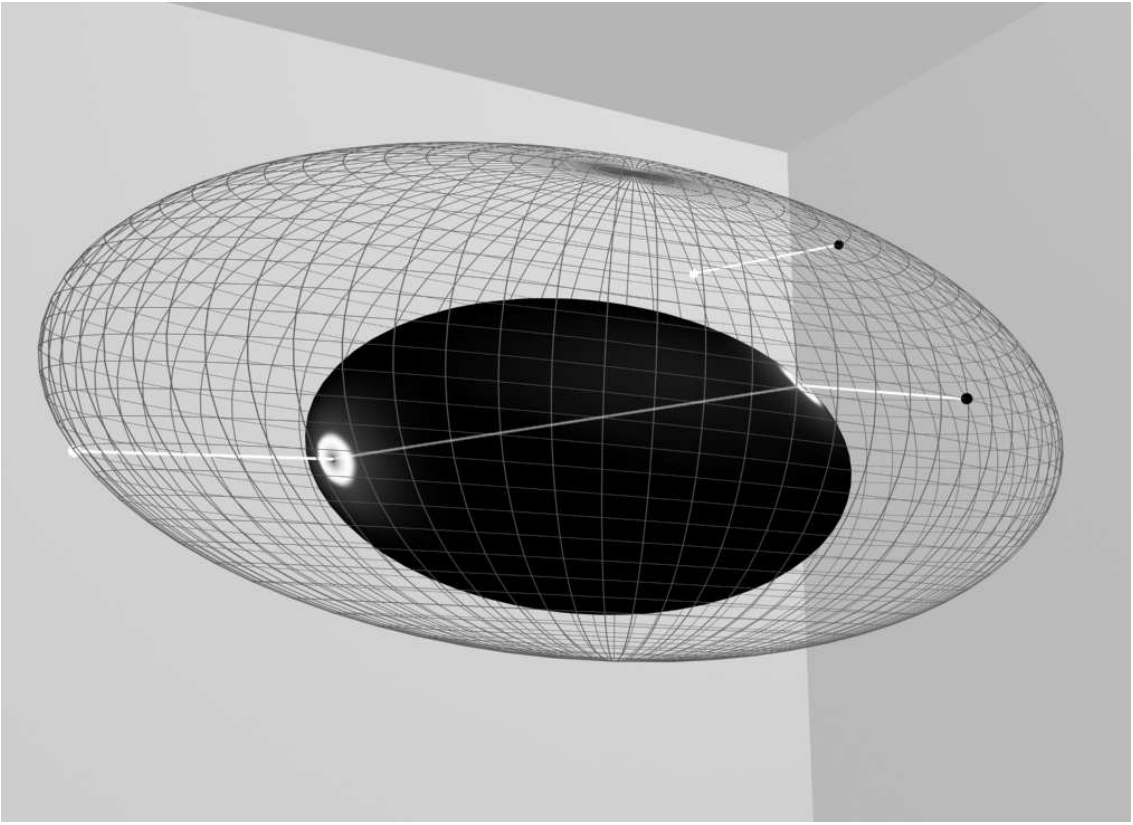}
\end{center}\caption[Illustration of a geodesic link]{Illustration of a geodesic link with $k=2$: the boundary of $\O$ is in wire, the one of $\o$ is black filled, the positive points are white,  the negative ones are black and a geodesic link is represented in white. The shaded off portion on the two penetration points gives indications about the $3D$-geometry of the geodesic link and of the inclusion. (Courtesy of Alexandre Marotta)}\label{F5.GeodesicLink}
      \end{figure}







\section{The main results}\label{S5.MainStatements}
We now state our main results. 

The first result (Theorem \ref{T5.Main1}) is an energetic estimation of $\inf_{H^1_g}F_\v$ in the spirit of  \cite{LR1}, \cite{LR2}, \cite{Sandier1} and \cite{BBM1}. This result is stated for the most general boundary condition: $g\in H^{1/2}(\p\O,\S^1)$.

The second theorem (Theorem \ref{T5.Main2}) is valid for the prepared boundary condition: $g\in\mathcal{H}$. This result expresses that for solutions of $\inf_{H^1_g}F_\v$, a concentration of the leading part of the energy occurs in a small neighborhood of the geodesic link of $\cl$ in $(\R^3,d_{a^2})$. 

The last result (Theorem \ref{T5.MainSymCase}) gives a more precise result than Theorems  \ref{T5.Main1}$\&$\ref{T5.Main2}  (a bounded error term instead of a $o(|\ln\v|)$-error term) under strong symmetric hypothesis.
\begin{thm}\label{T5.Main1}Let $g\in H^{1/2}(\p\O,\S^1)$. Then we have
\[
\inf_{H^1_g}{F_{\v}}=\pi L(g,d_{a^2})|\ln\v|+o(|\ln\v|).
\]
\end{thm}
\begin{thm}\label{T5.Main2}Let $g\in \mathcal{H}$ be s.t. $(\Pos\cup \Neg,d_{a^2})$ admits a unique geodesic link which is denoted $\cup_i{\Gamma_i}$.

Let $v_\v$ be a minimizer of $F_\v$ in $H^1_g$. Then the normalized energy density
\[
\mu_\v=\frac{\displaystyle\frac{U^2_\v}{2}|\n v_\v|^2+\frac{U^4_\v}{4\v^2}(1-|v_\v|^2)^2}{|\ln\v|}\,\mathscr{H}^3\text{ weakly converges in  $\O$ in the sense of the measures to }\pi a^2\mathscr{H}^1_{|\cup_i{\Gamma_i}} .
\]Here $\mathscr{H}^3$ is the $3$-dimensional Hausdorff measure and $\mathscr{H}^1_{|\cup_i{\Gamma_i}}$ is the one dimensional Hausdorff measure on ${\cup_i{\Gamma_i}}$.

In other words
\[
\forall \phi\in C^0(\O,\R)\cap L^\infty(\O,\R)\text{ we have } \int_{{\O}}{\phi\di\mu_\v}\to\pi\int_{\cup_i\Gamma_i}\phi a^2\di\mathscr{H}^1.
\]
\end{thm}
Note that this result gives a (uniform) energy concentration property of the minimizers along the geodesic link. Namely,  for all compact $K$ s.t. $K\cap\cup_i\Gamma_i=\emptyset$, we have $F_\v(v_\v,K)=o(|\ln\v|)$.

In order to obtain a more precise statement we assume that we are in a {\bf very} symmetrical case: $\O=B(0,1)$ and $\o=B(0,r_0)$, $r_0\in(0,1)$, $g\in\mathcal H$ is s.t. $\cl=\{p,n\}$ with $p=-n$. Under these hypotheses we have the following theorem.
\begin{thm}\label{T5.MainSymCase}
The following estimate holds
\[
\inf_{H^1_g}{F_\v}=\pi d_{a^2}(p,n)|\ln\v|+\mathcal{O}(1).
\]
Moreover, for all $\eta>0$, there is $C_\eta>0$ s.t. denoting $V_\eta=\{x\in\O\,|\,\dist(x,[p,n])\geq\eta\}$ and $v_\v$ a minimizer of  $F_\v$ in $H^1_g$, we have
\[
F_\v(v_\v,V_\eta)\leq C_\eta.
\]
\end{thm}
\section{Outline of the proofs}\label{S5.OutlineProofs}
The proofs of the above theorems strongly rely on the techniques developed in \cite{Sandier1}. The proofs of Theorems \ref{T5.Main1}, \ref{T5.Main2} and \ref{T5.MainSymCase} consist essentially into two parts devoted to obtaining respectively lower and  upper bounds.

The upper bound is obtained by the construction of a test function. The test function was obtained by Sandier in \cite{Sandier1} in the situation where $\O$ is a strictly convex domain and there is a geodesic link in $(\cl,d_{a^2})$ which is a union of line segments. In this special case, one may obtain (see Section \ref{S5.UpperBoundTheCaseOfSandier}): 
\begin{equation}\label{5.UpperBoundGeodesicareLines}
\inf_{ H^1_g}{F_{\v}}\leq\pi L(g,d_{a^2})|\ln\v|+\mathcal{O}(1).
\end{equation}

 For the general case, when the geodesic links are not unions of line segments, in \cite{BBM1}, Bourgain, Brezis and Mironescu adapted the construction of Sandier. In our case this leads to the bound: 
\begin{equation}\label{5.UpperBoundGeneralCase}
\inf_{ H^1_g}{F_{\v}}\leq\pi L(g,d_{a^2})|\ln\v|+o(|\ln\v|).
\end{equation}
(See Section \ref{S5.UpperBoundTheCaseBBM}.)

The lower bounds are obtained as in \cite{Sandier1}. The key ingredient is the construction of a "structure function" $\xi:\R^3\to\R$ (see Section \ref{S5.StructureFunction} for a precise definition). Due to the fact that for $M\in\R^3$, $x\mapsto \psi_M(x)=d_{a^2}(x,M)$ is not $C^1$ (its gradient is not continuous on $\p\o$ since $|\n \psi_M|=a^2$ in $\R^3\setminus\p\o$), we cannot obtain $\xi$ with exactly the same properties as in \cite{Sandier1} (see Proposition \ref{C5.Ustruc}). The consequence of this lack of smoothness of the distance function implies that our best lower bound is
\begin{equation}\label{5.LowerBoundGeneralCase}
\inf_{ H^1_g}{F_{\v}}\geq\pi L(g,d_{a^2})|\ln\v|-o(|\ln\v|).
\end{equation}

However, under strong symmetry hypotheses, namely, $\O=B(0,1),\o=B(0,r_0)$ and $\cl=\{p,n=-p\}$, the structure function $\xi$ enjoys additional properties (see Proposition \ref{P5.StructFonctSymCas}). In this symmetric case, one may obtain the sharper bound
\begin{equation}\label{5.LowerBoundSymmetricCase}
\inf_{ H^1_g}{F_{\v}}\geq\pi L(g,d_{a^2})|\ln\v|-\mathcal{O}(1).
\end{equation}

The estimate on $\inf_{H^1_g}F_\v$ in Theorem \ref{T5.Main1} (resp. Theorem \ref{T5.MainSymCase}) is a direct consequence of \eqref{5.UpperBoundGeneralCase}, \eqref{5.LowerBoundGeneralCase} (resp. \eqref{5.UpperBoundGeodesicareLines} and \eqref{5.LowerBoundSymmetricCase}) and of the  density of $\mathcal{H}$ in $H^{1/2}(\p\O,\S^1)$ (see Section \ref{S5.DensityArgument}) . 

Theorem \ref{T5.Main2} is proved along the main lines in \cite{Sandier1}.

Roughly speaking, under the hypotheses of  Theorem \ref{T5.Main2}, for all $x\in\O$, there is $\rho_x>0$ s.t. for $K=\overline{B(x,\rho_x)}$, one may consider a structure function $\xi$ adapted to $\cl$ which is constant in $K$ (see Section \ref{S5.StrFonCCompAAAA}). Arguing as in \cite{Sandier1}, if $K$ does not intersect the geodesic link, then we obtain that in $K$, a minimizer of $F_\v$ has its energy of order $o(|\ln\v|)$ (see \eqref{5.ConcentrationEnergySmallComp}). Thus $\mu$, the weak limit of $\mu_\v=\{{\frac{U^2_\v}{2}|\n v_\v|^2+\frac{U^4_\v}{4\v^2}(1-|v_\v|^2)^2}\}\mathscr{H}^3/{|\ln\v|}$ (which exists up to subsequence), is supported in $\overline{\O}\setminus K$. Therefore, one may prove that the support of $\mu$ is included in the geodesic link.

Otherwise, if $x$ is on the geodesic link, as explained in \cite{Sandier1}, then we obtain for $v_\v$ a minimizer and $\rho$ sufficiently small that (with $K=\overline{B(x,\rho)}$)
\[
\limsup\frac{F_\v(v_\v,K)}{|\ln\v|}\leq\pi {\rm long}_{a^2}(K\cap\cup_i{\Gamma_i}).
\]
 Theorem \ref{T5.Main2} is obtained by comparing $\mu$ to $\pi a^2\mathscr{H}^1_{|\cup_i{\Gamma_i}}$.
\section{First step in the proof of Theorem  \ref{T5.Main1}: an upper bound for ${\rm inf}_{H^1_g}F_\v$, $g\in\mathcal{H}$}\label{S5.thetwoUpperBounds}
\subsection{The case $\cl$ admits a geodesic link in $(\R^3,d_{a^2})$ which is a union of lines}\label{S5.UpperBoundTheCaseOfSandier}

 Assume that
 \begin{enumerate}[$\bullet$]
 \item $\O$ is strictly convex;
 \item there is $\Gamma=\cup\Gamma_i$, a geodesic link of $\cl$ in $(\R^3,d_{a^2})$, s.t. $\Gamma_i$ is a line segment for all $i$. 
 \end{enumerate}
 One may assume that the minimal connection associated to $\Gamma$ is  the identity and that $\Gamma_i$ is a geodesic curve between $p_i$ and $n_i$.
 
It is easy to see that
 \begin{enumerate}[$\bullet$]
 \item by strict convexity of $\O$ we have $\Gamma_i\subset\O$;
 \item from Remark \ref{R5.GeodLink}.2 we have for $i\neq j$ that $\Gamma_i\cap\Gamma_j=\emptyset$.
 \end{enumerate}
In this situation, we may mimic the construction of the test function made in Section 1 of  \cite{Sandier1}.

The test function is a fixed (independent of $\v$) $\S^1$-valued function outside $\mathcal{V}_\eta$, an $\eta$-tubular neighborhood of $\Gamma$ ($\eta$ small and independent of $\v$). 

Inside each tubular neighborhood $\mathcal{V}_{\eta,i}$ of $\Gamma_i$, the test function takes (essentially) the form (in an orthonormal basis $\{p_i,({\bf e}_x,{\bf e}_y,{\bf e}_z)\}$ s.t. $n_i=(0,0,|p_i-n_i|)$) 
\begin{equation}\label{5.ThetestFunctionOfSandier}
v_\v(x,y,z)=\begin{cases}
\displaystyle\frac{(x,y)}{|(x,y)|}\text{ if }\eta<z<|p_i-n_i|-\eta\text{ and $\v<|(x,y)|<\eta$}
\\\displaystyle\frac{(x,y)}{\v}\text{ if }\eta<z<|p_i-n_i|-\eta\text{ and $|(x,y)|<\v$}
\\\displaystyle\frac{1}{2}\int_{\tilde{\mathcal{V}}_{\eta,i}}\left\{|\n v_\v|^2+\frac{1}{2\v^2}(1-|v_\v|^2)^2\right\}\leq2\pi\eta|\ln\v|+C(g,\eta)\\\text{with }\tilde{\mathcal{V}}_{\eta,i}=\mathcal{V}_{\eta,i}\cap\{0<z<\eta\text{ or }|p_i-n_i|-\eta< z<|p_i-n_i|\}
\end{cases}.
\end{equation}
(see \cite{Sandier1} for more details).

From the strict convexity of $\o$, for all line $D\subset \R^3$, we have
\[
{\rm long}\left(D\cap\{x\in\R^3\,|\,\dist(x,\p\o)\leq\sqrt\v\}\right)\leq C\sqrt\v\text{ with $C>0$ independent of $\v$}. 
\]
Thus one may obtain from Proposition \ref{P5.Fond1}.4. that \eqref{5.UpperBoundGeodesicareLines} holds.
\subsection{The general case}\label{S5.UpperBoundTheCaseBBM}
One may adapt  the above construction to a more general situation: the geodesic links are not unions of line segments and $\O$ is convex (not strictly convex). Without loss of generality we may assume that $\sigma={\rm Id}$ is a minimal connection.

For the standard Ginzburg-Landau energy, this has been done in \cite{BBM1}; there, $\O$ is not supposed convex. Roughly speaking, the argument there consists in replacing in Sandier's proof, line segments by curves.

Their construction begins with 
\begin{itemize}
\item[$\bullet$]the modification of $\O$ (flattening $\p\O$ close to the singularities) 
\item[$\bullet$]for $\eta>0$, we construct an approximate (smooth) geodesic link $\Gamma_\eta=\cup_i\Gamma_i^\eta\subset\O$ s.t. 
\begin{itemize}\item ${\rm long}_{a^2}(\Gamma_\eta)\leq {\rm long}_{a^2}(\Gamma)+\eta$,
\item $\Gamma$ is a geodesic link and $\Gamma_i^\eta$ is an approximate geodesic curve between $p_i$ and $n_i$,
\item $\Gamma_i^\eta\cap\Gamma_j^\eta=\emptyset$ for $i\neq j$ and $\Gamma_i^\eta$ is orthogonal to $\p\O$ at $p_i$ and $n_i$.
\end{itemize}
\end{itemize}
In order to be applicable to our situation, this construction requires the additional property about its intersection with a small neighborhood of $\p\o$:
\[\mathscr{H}^1(\Gamma_\eta\cap\{\dist(x,\p\o)<\sqrt\v\})=\mathcal{O}(\sqrt\v).
\]
We can clearly find $\Gamma_\eta$ satisfying this property.

By adapting the construction of $v_\v$ in \eqref{5.ThetestFunctionOfSandier}, one may construct a test function $v_\v^\eta$ having $\Gamma_\eta$ as set of zeros and satisfying, for each $\eta>0$, the estimate 

\[
\frac{\inf_{v\in H^1_g}{F_{\v}(v)}}{|\ln\v|}\leq \frac{F_\v(v^{\eta}_\v)}{|\ln\v|}+o_\v(1)=\pi{\rm long}_{a^2}(\Gamma_\eta)+o_\v(1)\leq \pi{\rm long}_{a^2}(\Gamma)+\eta+o_\v(1).
\]
In order to obtain this estimate we rely on the formula of $v_\v^\eta$, Proposition \ref{P5.Fond1}.4. and the assumption $\mathscr{H}^1(\Gamma_\eta\cap\{\dist(x,\p\o)<\sqrt\v\})=\mathcal{O}(\sqrt\v)$.

Consequently we deduce that \eqref{5.UpperBoundGeneralCase} holds.

\section{Structure functions}\label{S5.StructureFunction}
For $g\in\mathcal{H}$, we construct suitable structure functions adapted to the singularities of $g$. 

Very roughly speaking, a structure function $\xi\in {\rm Lip}(\R^3,\R)$ is a smooth map whose restriction on $\p\O$ almost maximizes \eqref{5.DeflongMin1}. More qualitative properties of $\xi$ will be described in Propositions \ref{C5.Ustruc}, \ref{C5.StructureFunctionCompact}, \ref{C5.StructureFunctionCompactBisBis} and  \ref{P5.StructFonctSymCas}.

Structure functions are used to compute a lower bound of the energy. We may sketch the proof of a (sharp) lower bound of $F_\v(v_\v)$:
\begin{itemize}
\item We integrate by layers (on hypersurfaces) using the Coarea formula; 
\item  The layers are the level sets of the structure function $\xi:\R^3\to\R$. Lower bounds on each hypersurfaces are obtained via a standard result on Ginzburg-Landau energy (Proposition \ref{P5.lemmedeSandier}). Because the use of Coarea formula adds an extra-weight $\frac{1}{|\n \xi|}$ on the layers, in order to get a lower bound we drop the weights $U_\v^2$ and $U_\v^4\geq b^2 U_\v^2$ by taking $\xi$ s.t. $\frac{U_\v^2}{|\n\xi|}\apprge 1$. Therefore, an important property of $\xi$ is that its gradient is controlled by $U_\v^2$ ($|\n\xi|\lesssim U_\v^2$).
\item In order to treat the integrals on $\{\xi=t\}$ we use a standard lower bound (Proposition \ref{P5.lemmedeSandier}) which need a control on the second fundamental form of $\{\xi=t\}$. This is an another property of the structure function.
\item  The "quantity" of layers is related with the length of a minimal connection $L(\cl,d_{a^2})$. In order to get a sharp lower bound we need that $\sum_i\{\xi(p_i)-\xi(n_i)\}\apprge L(\cl,d_{a^2})$; it is another important property of $\xi$.
\end{itemize}
We present below four structure functions corresponding to four different settings.


\subsection[Second step in the proof of Theorem \ref {T5.Main1}]{Second step in the proof of Theorem \ref {T5.Main1}:  a structure function}
This subsection is devoted to the presentation of a structure function adapted to $\cl$ and $a$. This function is given in the following.


\begin{prop}\label{C5.Ustruc}
For all $\eta>0$, there is $C_{\eta}>0$, $E_{\eta}\subset\R$, $\xi_{\eta}\in C^\infty(\R^3,\R)$ and $\v_\eta>0$ s.t. for $0<\v<\v_\eta$, 
\begin{enumerate}[1.]
\item $|\n\xi_{\eta}|\leq \min(a^2,U_\v^2+\v^4)$ in $\R^3$,
\item $\sum_{i\in\N_k}\left\{\xi_{\eta}(p_i)-\xi_{\eta}(n_i)\right\}\geq L(\cl,d_{a^2})-\eta$,
\item $\mathscr{H}^1(E_{\eta})\leq\eta$ and for all $t\in\R\setminus E_{\eta}$, $\{\xi_{\eta}=t\}$ is a closed hypersurface whose second fundamental form is bounded by $C_{\eta}$.
\end{enumerate}
\end{prop}
This proposition is proved in Section \ref{S5.ProofPremierresmlsqkdfgn}
\begin{remark}
In Proposition \ref{C5.Ustruc}, the structure function $\xi_\eta$ does not dependent on (small) $\v$. The structure functions given by Propositions \ref{C5.StructureFunctionCompact} and \ref{C5.StructureFunctionCompactBisBis} (below) are also independent of small $\v$. On the other hand, the structure function given by Proposition \ref{P5.StructFonctSymCas} (below) depends on $\v$.
\end{remark}
\subsection{The proof of Proposition \ref{C5.Ustruc}}\label{S5.ProofPremierresmlsqkdfgn}
\subsubsection{Some definitions}

Let $\delta_0=\displaystyle10^{-2}\min\{1;\dist(\p\o,\p\O)\}$. For  $0<\delta\leq\delta_0$, we define $\o_\delta:=\o+B(0,\delta)$, $\alpha_0=a^2$ and
\[
\begin{array}{cccc}\alpha_\delta:&\R^3&\to&\{1,b^2\}\\&x&\mapsto&\begin{cases}b^2&\text{if }x\in\o_\delta\\1&\text{otherwise}\end{cases}\end{array}.
\]

For $x,y\in\R^3$ and $0\leq\delta<\delta'\leq\delta_0$, we have 
\begin{equation}\label{5.EstiPourBougIncl}
d_{\alpha_{\delta'}}(x,y)\leq d_{\alpha_\delta}(x,y)\leq d_{\alpha_{\delta'}}(x,y)+\mathcal{O}(|\delta'-\delta|).
\end{equation} 
The first inequality is a direct consequence of $\alpha_{\delta'}\leq\alpha_{\delta}$. We prove the second inequality. Consider $x,y\in\R^3$ s.t. $d_{\alpha_{\delta'}}(x,y)< d_{\alpha_{\delta}}(x,y)$. We obtain that if $\Gamma$ is a  geodesic joining $x$ and $y$ in $(\R^3,d_{\alpha_{\delta'}})$, then we have $\Gamma\cap\p\o_{\delta'}\neq\emptyset$. 

Note that by Proposition \ref{P5.ClassificationGeodesiquespecialcase}, we have ${\rm Card}(\Gamma\cap\p\o_{\delta'})\in\{1,2\}$.

Assume that $\Gamma\cap\p\o_{\delta'}=\{x',y'\}$ with $d_{\alpha_{\delta'}}(x,x')<d_{\alpha_{\delta'}}(x,y')$. The situation where $\Gamma\cap\p\o_{\delta'}=\{z\}$ is similar.

Consider $x''=\Pi_{\overline{\o_{\delta}}}(x')$ and $y''=\Pi_{\overline{\o_{\delta}}}(y')$. Here $\Pi_{\overline{\o_{\delta}}}$ stands for the orthogonal projection on ${\overline{\o_{\delta}}}$. By the definition of $x''$ and $y''$ we have $\deucl(x',x'')=\deucl(y',y'')=\delta'-\delta$. By Proposition \ref{P5.ClassificationGeodesiquespecialcase}, we deduce that $d_{\alpha_{\delta'}}(x'',y'')=d_{\alpha_{\delta}}(x'',y'')$.

Since $x',y'\in\Gamma$, we have 
\begin{eqnarray*}
d_{\alpha_{\delta'}}(x,y)&=&d_{\alpha_{\delta'}}(x,x')+d_{\alpha_{\delta'}}(x',y')+d_{\alpha_{\delta'}}(y',y)
\\&\geq&d_{\alpha_{\delta'}}(x,x')+d_{\alpha_{\delta'}}(x'',y'')+d_{\alpha_{\delta'}}(y',y)-2|\delta'-\delta|
\\&\geq&d_{\alpha_{\delta}}(x,x')+d_{\alpha_{\delta}}(x'',y'')+d_{\alpha_{\delta}}(y',y)-2|\delta'-\delta|
\\&\geq&d_{\alpha_{\delta}}(x,x')+d_{\alpha_{\delta}}(x',y')+d_{\alpha_{\delta}}(y',y)-4|\delta'-\delta|
\\&\geq&d_{\alpha_{\delta}}(x,y)-4|\delta'-\delta|
\end{eqnarray*}
Consequently, \eqref{5.EstiPourBougIncl} holds.

Thus, for $\cl=\Pos\cup \Neg$ defined above, we obtain that 
\begin{equation}\label{5.ContinuityOfMinimalConnexion}
L(\cl,d_{\alpha_\delta})= L(\cl,d_{\alpha_{\delta'}})+\mathcal{O}(|\delta'-\delta|).
\end{equation}
\begin{remark}\label{R5.Uniqlkqjsdfhsqkdlfj}
Using \eqref{5.ContinuityOfMinimalConnexion}, it is easy to check that there is $\delta_{b,\o,\cl}<\displaystyle\delta_0$ s.t. for $\delta< \delta_{b,\o,\cl}$:
 If $\sigma$ is a minimal connection  in $(\cl,d_{\alpha_\delta})$, then $\sigma$ is a minimal connection  in $(\cl,d_{\alpha_0})$.
\end{remark}

We now state a preliminary result: Proposition \ref{P5.1}. Proposition \ref{C5.Ustruc} will be a direct consequence of Proposition \ref{P5.1}. 
\begin{prop}\label{P5.1}
For $\eta>0$ there is $\delta_\eta>0$ s.t. for $\delta_\eta>\delta>0$ there are $C_{\eta,\delta}>0$, $E_{\eta,\delta}\subset\R$ and $\xi_{\eta,\delta}\in C^\infty(\R^3,\R)$ s.t.
\begin{enumerate}
\item $|\n\xi_{\eta,\delta}|\leq \alpha_\delta$ in $\R^3$
\item $\sum_{i\in\N_k}\left\{\xi_{\eta,\delta}(p_i)-\xi_{\eta,\delta}(n_i)\right\}\geq  L(\cl,d_{\alpha_\delta})-\eta$
\item $\mathscr{H}^1(E_{\eta,\delta})\leq\eta$ and for all $t\in\R\setminus E_{\eta,\delta}$, $\{\xi_{\eta,\delta}=t\}$ is a closed two dimensional  surface with its second fundamental form which is bounded by $C_{\eta,\delta}$.
\end{enumerate}
\end{prop}
\begin{proof}[Proof of Proposition \ref{C5.Ustruc}]
Let $\eta>0$ and fix $0<\delta<\delta_{\eta/2}$ ($\delta_{\eta/2}$ given by Proposition \ref{P5.1}) s.t. 
\[
L(\cl,d_{\alpha_\delta})+\frac{\eta}{2}\geq L(\cl,d_{a^2}) \text{ (using \eqref{5.ContinuityOfMinimalConnexion})}.
\] 
Consider $\v_\eta>0$ s.t. for $0<\v<\v_\eta$ we have 
\[
C\e^{-\gamma \delta/\v}<\v^4\,(C\text{ and }\gamma\text{ are given by \eqref{5.Uepsprochea}}).
\]
We take $\xi_\eta=\xi_{\eta/2,\delta}$ obtained from Proposition \ref{P5.1}.

Clearly, $\xi_\eta$ satisfies {\it 2.} and {\it 3.} with $E_\eta=E_{\eta/2,\delta}$ and $C_\eta=C_{\eta/2,\delta}$. 

It is direct to obtain that 
\[
|\n\xi_\eta|-U_\v^2\leq  \alpha_\delta-U_\v^2\leq\begin{cases}b^2-U_\v^2\leq0&\text{if }\dist(x,\o)<\delta\\\v^4&\text{otherwise}\end{cases}.
\]
It follows that $\xi_\eta$ satisfies {\it 1} since $\alpha_\delta\leq a^2$.
\end{proof}

\subsubsection{The proof of Proposition \ref{P5.1}}\label{Proofwithoutcompact}

We construct $\xi_{\eta,\delta}$ in five steps.\vspace{2mm}
\\Let $\eta>0$ and $0<\delta<\delta'<\delta_{b,\o,\cl}$ (here $\delta_{b,\o,\cl}$ is defined in Remark \ref{R5.Uniqlkqjsdfhsqkdlfj}). We denote $\alpha=\alpha_{\delta}$ and $\alpha'=\alpha_{\delta'}$.   Assume that $\Pos=\{p_1,...,p_k\}$ and $\Neg=\{n_1,...,n_k\}$ are s.t. $\sigma={\rm Id}$ is a minimal connection in $(\cl,d_{\alpha'})$.\vspace{2mm}
\\{\bf Step 1: }There is $\xi_0:\cl\to\R$ s.t. $\xi_0$ is $1$-Lipschitz in $(\cl,d_{\alpha'})$ and $\xi_0(p_i)-\xi_0(n_i)=d_{\alpha'}(p_i,n_i)$
\vspace{2mm}
\\This step is a direct consequence of Lemma 4.2 in \cite{BCL1} (see also Lemma 2.2 in \cite{Sandier1} or Lemma 2 in \cite{BrezisCristal}).
\vspace{2mm}
\\{\bf Step 2:} We extend $\xi_0$ to $\R^3$: there is some $\xi_1\in\Lip(\R^3,\R)$ s.t. $|\n\xi_1|=\alpha'$ and $\xi_{1|\cl}\equiv\xi_0$
\vspace{2mm}
\\Although the argument is the same as in \cite{Sandier1}, for the convenience of the reader, we recall the construction.

Consider 
\[
\xi_1(x)=\max_i\left\{\xi_0(p_i)-d_{\alpha'}(x,p_i)\right\},\,x\in\R^3.
\]
Then we have the following.
\begin{enumerate}[$\bullet$]
\item If $k=1$, $\xi_1(x)=\xi_0(p_1)-d_{\alpha'}(x,p_1)$ is s.t. $\xi_1\in\Lip(\R^3,\R)$ s.t. $|\n\xi_1|=\alpha'$ and $\xi_{1|\cl}\equiv\xi_0$. We now assume that $k\geq2$.
\item $\xi_{1|\cl}\equiv\xi_0$: let $M\in \cl$ and $i$ be s.t. $M\in\{p_i,n_i\}$ 
 and $j\neq i$, it is clear that
\[
\xi_0(p_i)-d_{\alpha'}(M,p_i)-\xi_0(p_j)+d_{\alpha'}(M,p_j)=\begin{cases}\xi_0(p_i)-\xi_0(p_j)+d_{\alpha'}(p_i,p_j)\geq0&\text{if }M=p_i\\\xi_0(n_i)-\xi_0(p_j)+d_{\alpha'}(n_i,p_j)\geq0&\text{if }M=n_i\end{cases};
\]
\item $|\n\xi_1|= \alpha'$: for all $i$ we have (Proposition \ref{P5.ClassificationGeodesiquespecialcase}.5)
\[
\left|\n\left[\xi_0(p_i)-d_{\alpha'}(x,p_i)\right]\right|=\left|\n d_{\alpha'}(x,p_i)\right|=\alpha'\text{ in }L^\infty(\R^3).
\]
\end{enumerate}
\vspace{2mm}
{\bf Step 3:}  We construct a smooth approximation: let $0<\beta<\delta$, there is $\xi_2\in C^\infty(\R^3,\R)$ s.t. $|\n\xi_2|\leq (1-\beta) \alpha$  and 
\begin{equation}\label{5.DisplayForm1}
\sum_{i\in\N_k}\left\{\xi_2(p_i)-\xi_2(n_i)\right\}\geq L(\cl,d_{\alpha})-\eta/2
\end{equation}
Let $(\rho_t)_{\delta'-\delta>t>0}$ be a classical mollifier, namely $\rho_t(x)=t^{-3}\rho(x/t)$ with $\rho\in C^\infty(\R^3,[0,1])$, ${\rm Supp}\,\rho\subset B(0,1)$ and $\int_{\R^3}\rho=1$. 

Consider 
\[
\xi_2(x):=(1-\beta)\xi_1\ast\rho_t(x).
\]
Because for $x\in\o_{\delta}$ we have $x+B(0,t)\subset \o_{\delta'}$, we have $|\n\xi_2|\leq (1-\beta) \alpha$:
\[
\frac{|\n\xi_2(x)|}{(1-\beta)}=|[\n \xi_1\ast\rho_t](x)|\leq
\begin{cases} 
1&\text{if $x\notin \o_\delta$}
\\
[|\n \xi_1|\ast\rho_t](x)=[\alpha'\ast\rho_t](x)= b^2&\text{if $x\in \o_\delta$}
\end{cases}.
 \] 
 Moreover Condition \eqref{5.DisplayForm1} is satisfied:
 \[
 \sum_{i\in\N_k}\left\{\xi_2(p_i)-\xi_2(n_i)\right\}\geq (1-\beta) \left[\sum_{i\in\N_k}\left\{\xi_1(p_i)-\xi_1(n_i)\right\}-\mathcal{O}(\delta')\right].
 \]
 Thus letting $\delta_\eta$ sufficiently small s.t. for $0<\beta<\delta<\delta'<\delta_\eta$ we have 
 \[
 (1-\beta) \left[\sum_{i\in\N_k}\left\{\xi_1(p_i)-\xi_1(n_i)\right\}-\mathcal{O}(\delta')\right]\geq L(\cl,d_{\alpha'})-\eta/3\geq L(\cl,d_{\alpha})-\eta/2.
 \]
 Therefore $\xi_2$ has the desired properties for $\beta<\delta<\delta_\eta$.
\\
\vspace{2mm}\\{\bf Step 4:} Let $\tilde{\O}$ be a smooth and bounded open neighborhood of $\overline{\O}$. We approximate $\xi_2$ by a Morse function $\xi_{\eta,\delta}$ in $C^1(\tilde{\O})$. Function $\xi_{\eta,\delta}$ is s.t. we have $\xi_{\eta,\delta}\in C^\infty(\R^3,\R)$ and
\[
\|\xi_{\eta,\delta}-\xi_2\|_{L^\infty(\tilde{\O})}\leq\eta/(4k),
\]
\[
|\n\xi_{\eta,\delta}|\leq \alpha,
\]
\[
\xi_{\eta,\delta}\text{ is a Morse function},
\]
\[
\exists R=R({\eta,\delta})>0\text{ s.t. in }\R^3\setminus B(0,R),\,\xi_{\eta,\delta}=|x|/2
\]
Clearly $\xi_{\eta,\delta}$ satisfies {\it1}. and {\it2}. of Proposition \ref{P5.1}. 
\vspace{2mm}
\\{\bf Step 5:} We follow \cite{Sandier1}. We construct $E_{\eta,\delta}$  
\vspace{2mm}\\
Let $\{x_1,...,x_l\}$ be the set of the critical points of $\xi_{\eta,\delta}$. Then there is $C=C({\eta,\delta})>0$ s.t.:
\[
\inf_{\overline{B(0,R)}\setminus\cup_i B(x_i,\rho)}|\n\xi_{\eta,\delta}|\geq\frac{\rho}{C}\text{ since the critical points are not degenerate}
\]
and
\[
\mathscr{H}^1\left[\xi_{\eta,\delta}(\cup_iB(x_i,\rho))\right]\leq C\rho.
\]
We consider  $\rho>0$ s.t. $C\rho\leq\eta$ and set $E_{\eta,\delta}=\xi_{\eta,\delta}(\cup_iB(x_i,\rho))$.

For $t\notin E_{\eta,\delta}$, we have 
\begin{enumerate}[$\bullet$]
\item if $x\in \{\xi_{\eta,\delta}=t\}\setminus B(0,R)$, then the second fundamental form of   $\{\xi_{\eta,\delta}=t\}$ in $x$ is bounded by $C'_{\eta,\delta}$,
\item if $x\in\{\xi_{\eta,\delta}=t\}\cap B(0,R)$, then the second form is bounded by $\displaystyle C''_{\eta,\delta}=\frac{C\sup_{\overline{B(0,R)}}{|D^2\xi_{\eta,\delta}|}}{\inf_{\overline{B(0,R)}\setminus\cup_i B(x_i,\rho)}|\n\xi_{\eta,\delta}|}.$
\end{enumerate}
We find that the second fundamental form is globally bounded by $C_{\eta,\delta}=\max\{C_{\eta,\delta}',C_{\eta,\delta}''\}$.
\subsection[Second step in the proof of Theorem \ref{T5.Main2}]{Second step in the proof of Theorem \ref {T5.Main2}: two structure functions}\label{S5.StrFonCCompAAAA}

 Letting $\mu_\v=\frac{\displaystyle\frac{U^2_\v}{2}|\n v_\v|^2+\frac{U^4_\v}{4\v^2}(1-|v_\v|^2)^2}{|\ln\v|}\,\mathscr{H}^3$, because by \eqref{5.UpperBoundGeneralCase} $\mu_\v$ is bounded, we get that, up to pass to a subsequence, $\mu_\v$ converges in the sense of the measures to $\mu$. 
 
 The proof of Theorem \ref{T5.Main2} consists in the identification of $\mu$. We want to prove that $\mu=\pi a^2\H^1_{|\cup\Gamma_i}$ with $\cup\Gamma_i$ is the unique the geodesic link joining the singularities of $g$. Because by Theorem \ref {T5.Main1} we have that $\mu$ and $\pi a^2\H^1_{|\cup\Gamma_i}$ has the same total mass, it suffices to prove that for all $x_0\in\O$ there is $r_{x_0}>0$ s.t. $\mu(B(x_0,r_{x_0}))\leq \pi a^2\H^1_{|\cup\Gamma_i}(B(x_0,r_{x_0}))$. Therefore we have to obtain lower bound for $F_\v(v_\v,\O\setminus\overline{B(x_0,r_{x_0})})/|\ln\v|$.
  
 In order to get a suitable lower bound to prove Theorem \ref {T5.Main2} we need two kinds of structure function:
\begin{itemize}
\item[$\bullet$] a structure function which is constant in a neighborhood $K=\overline{B(x_0,r_{x_0})}$ of an arbitrary point $x_0$ which is not on a geodesic link joining the singularities;
\item[$\bullet$] a structure function which is constant in a neighborhood $K=\overline{B(x_0,r_{x_0})}$ of an arbitrary point $x_0$ which is on {\bf the} geodesic link joining the singularities (assuming the geodesic link is unique). 
\end{itemize}
We briefly sketch the strategy used in Section \ref{S5.LowerBoundSandier} to get that $\mu=\pi a^2\H^1_{|\cup\Gamma_i}$. Using the upper bound \eqref{5.UpperBoundGeneralCase}, if $x_0$  is not on a geodesic link joining the singularities, then we get the lower bound (letting $K=\overline{B(x_0,r_{x_0})}$)
\begin{equation}\label{Aletlkqsdjfqskldjhfkshdf}
F_\v(v_\v,\O\setminus K)\geq F_\v(v_\v,\O)-o(|\ln\v|).
\end{equation}
From the estimate \eqref{Aletlkqsdjfqskldjhfkshdf} we deduce that $F_\v(v_\v,K)=o(|\ln\v|)$. Therefore $\mu$ is supported in the union of the geodesic links joining the singularities.
 
 Assuming that the geodesic link is unique (we denote this geodesic link by $\cup\Gamma_i$), we get a structure function which allows to obtain for $x_0$ which is on {\bf the} geodesic link and small $r_{x_0}$ the estimate
 \begin{equation}\label{mpfkjfjfjhfngngng}
 F_\v(v_\v,K)\leq \begin{cases}2\pi a(x_0)^2r_{x_0}|\ln\v|+o(|\ln\v|)&\text{if }x_0\notin\p\o\\\pi(1+b^2)r_{x_0}|\ln\v|+o(|\ln\v|)&\text{if }x_0\in\p\o\end{cases}.
 \end{equation}
 Therefore, we get that $\mu\leq\pi a^2\H^1_{|\cup\Gamma_i}$ and thus we obtain that $\mu=\pi a^2\H^1_{|\cup\Gamma_i}$.

In order to get estimates \eqref{Aletlkqsdjfqskldjhfkshdf} and \eqref{mpfkjfjfjhfngngng} we need two kinds of structure functions.

In the spirit of Proposition \ref{C5.Ustruc}, to get that the support of $\mu$ is included in the geodesic links (Estimate \eqref{Aletlkqsdjfqskldjhfkshdf}), we have the following structure function (Proposition \ref{C5.StructureFunctionCompact} does not need the hypothesis that there exists a unique geodesic link)
\begin{prop}\label{C5.StructureFunctionCompact}
Let $x_0\in\R^3$ be s.t. $x_0$ is not on a geodesic link of $\cl$ in $(\R^3,d_{a^2})$. There is $r_{x_0}>0$ s.t. denoting $K=\overline{B(x_0,r_{x_0})}$, for $\eta>0$ there are $\xi_{\eta,K}\in C^\infty(\R^3,\R)$, $E_{\eta,K}\subset\R$, $C_{\eta,K}>0$ and $\v_{\eta,K}>0$ s.t. for $0<\v<\v_{\eta,K}$, 
\begin{enumerate}[1.]
\item $|\n\xi_{\eta,K}|\leq \min(a^2,U_\v^2+\v^4)$ in $\R^3$ and $\xi_{\eta,K}$ is constant in $K$
\item $\sum_{i\in\N_k}\{\xi_{\eta,K}(p_i)-\xi_{\eta,K}(n_i)\}\geq L(\cl,d_{a^2})-\eta$
\item $\mathscr{H}^1(E_{\eta,K})\leq\eta$ and for all $t\in\R\setminus E_{\eta,K}$, $\{\xi_{\eta,K}=t\}$ is a closed hypersurface whose second fundamental form is bounded by $C_{\eta,K}$.
\end{enumerate}
\end{prop}
And in order to compare $\mu$ with $\pi a^2\H^1_{|\cup\Gamma_i}$ (Estimate \eqref{mpfkjfjfjhfngngng}) we have the following structure function
\begin{prop}\label{C5.StructureFunctionCompactBisBis}
Assume that there exists a unique geodesic link $\cup\Gamma_i$ joining the singularities of $g$  in $(\R^3,d_{a^2})$.

Let $x_0\in \cup\Gamma_i\setminus\cl$, then  there is $r_{x_0}>0$ s.t. denoting $K=\overline{B(x_0,r_{x_0})}$, for $\eta>0$ there are $\xi_{\eta,K}\in C^\infty(\R^3,\R)$, $E_{\eta,K}\subset\R$, $C_{\eta,K}>0$ and $\v_{\eta,K}>0$ s.t. for $0<\v<\v_{\eta,K}$, 
\begin{enumerate}[1.]
\item $|\n\xi_{\eta,K}|\leq \min(a^2,U_\v^2+\v^4)$ in $\R^3$ and $\xi_{\eta,K}$ is constant in $K$
\item \begin{tabular}{ccc}$\sum_{i\in\N_k}\{\xi_{\eta,K}(p_i)-\xi_{\eta,K}(n_i)\}$&$\geq$&$ L(\cl,d_{a^2})-a^2\H^1_{|\cup\Gamma_i}(K)-\eta$\\&$=$&$\begin{cases} L(\cl,d_{a^2})-2a^2(x_0)r_{x_0}-\eta&\text{if }x\notin\p\o\\ L(\cl,d_{a^2})-(1+b^2)r_{x_0}-\eta&\text{if }x\in\p\o\end{cases}$\end{tabular}
\item $\mathscr{H}^1(E_{\eta,K})\leq\eta$ and for all $t\in\R\setminus E_{\eta,K}$, $\{\xi_{\eta,K}=t\}$ is a closed hypersurface whose second fundamental form is bounded by $C_{\eta,K}$.
\end{enumerate}
\end{prop}
The proofs of Propositions \ref{C5.StructureFunctionCompact} and \ref{C5.StructureFunctionCompactBisBis} are given in Section \ref{S5.ProofDeuxEtTroisStructFunc}.
\subsection{Proof of Propositions \ref{C5.StructureFunctionCompact} and \ref{C5.StructureFunctionCompactBisBis} }\label{S5.ProofDeuxEtTroisStructFunc}
In order to prove Propositions \ref{C5.StructureFunctionCompact} $\&$ \ref{C5.StructureFunctionCompactBisBis}, we follow the proof of Proposition \ref{C5.Ustruc} using a special pseudo metric taking care of the compact set $K$.

\subsubsection{Definition and properties of a special pseudometric}
Let $f:\R^3\to[b^2,1]$ be a Borel function and let $\emptyset\neq K\subset\R^3$ be a smooth compact set. We define 
\[
d_f^K(x,y)=\min\left\{d_f(x,y),d_f(x,K)+d_f(y,K)\right\}.
\]
Here $d_f(x,K)=\min_{y\in K} d_{f}(x,y)$.

Then $d^K_f$ is a pseudometric in $\R^3$. If, in addition $K\cap\cl=\emptyset$, then $d^K_f$ is a distance in $\cl$.  Therefore the minimal connection of $\cl$ and the length of a minimal connection $L(\cl,d^K_f)$ with respect to $d_f^K$ make sense.

Clearly, if $x,y\in\R^3$, then we have $d_f^K(x,y)=0\Leftrightarrow x=y\text{ or }x,y\in K$. One may easily prove that
\begin{equation}\label{5.ContinuityOfDistanceCompact}
d^K_f(x,y)\leq d_f(x,y)\leq d_f^K(x,y)+\diam(K).
\end{equation}

We are interested in the special case  $K=\overline{B(x_0,r)}$ for some $x_0\in\O$ and $f=\alpha_\delta$ with $\delta\in[0,\delta_0]$, $\delta_0=\displaystyle10^{-2}\min\{1;\dist(\p\o,\p\O)\}$. 

Note that we have a similar estimate to \eqref{5.ContinuityOfMinimalConnexion}, namely for $0\leq\delta<\delta'<\delta_0$ 
\begin{equation}\label{5.ContinuityOfMinimalConnexionCompact}
L(\cl,d^K_{\alpha_\delta})= L(\cl,d^K_{\alpha_{\delta'}})+\mathcal{O}(|\delta'-\delta|).
\end{equation}
From \eqref{5.ContinuityOfDistanceCompact}, \eqref{5.ContinuityOfMinimalConnexionCompact} and as in Remark \ref{R5.Uniqlkqjsdfhsqkdlfj} we have the following.
\begin{remark}\label{R5.Uniqlkqjsd}
There is $\tilde{\delta}_{b,\o,\cl}<\displaystyle10^{-2}\cdot\dist(\p\o,\p\O)$ s.t. for $x_0\in\R^3$, $0<\delta,r< \tilde{\delta}_{b,\o,\cl}$ and $K=\overline{B(x_0,r)}$ s.t. $K\cap\cl=\emptyset$:
 If $\sigma$ is a minimal connection  in $(\cl,d^K_{\alpha_\delta})$, then $\sigma$ is a minimal connection  in $(\cl,d_{\alpha_0})$.
\end{remark}
\begin{definition} For $y\notin K$ and $x\in\R^3$, we say that 
\begin{enumerate}[$\bullet$]
\item $\Gamma$ is a {\it $K$-curve joining $x,y$} if $\Gamma$ is a finite union of curves included in $\R^3\setminus K$ (except perhaps their endpoints) s.t.: 
\begin{enumerate}[-]
\item their endpoints are either $x$ or $y$ or an element of $\p K$,
\item if $z\in\{x,y\}\setminus K$, then $z$ is an endpoint of a connected component of $\Gamma$.
\end{enumerate}
\item $\Gamma$ is a {\it minimal $K$-curve joining $x,y$} if $\Gamma=\cup_i \Gamma_i$ is a $K$-curve joining $x,y$, where the $\Gamma_i$'s are disjoint curves and $\sum_i{\rm long}_{a^2}(\Gamma_i)=d^K_{a^2}(x,y)$.
\end{enumerate} 
\end{definition}
We next sum up the main properties of $d^K_{a^2}$.
\begin{prop}\label{P5.PropertiesdK}
Let $x_0\in\R^3$, $r>0$ and $K=\overline{B(x_0,r)}$. 
\begin{enumerate}[1.]
\item If $y\notin K$ then for all $x\in\R^3$ there is a minimal $K$-curve joining $x,y$. Moreover, a minimal $K$-curve is the union of at most two geodesics in $(\R^3,d_{a^2})$.
\item For $x,y\in\R^3$, $x\neq y$ and $x_0\neq x,y$ we have: 
\begin{enumerate}[i.]
\item If $x_0\in\R^3\setminus\p\o$ and $x_0$ is on a geodesic joining $x,y$ in $(\R^3,d_{a^2})$, then there is $r_{x_0,x,y}>0$ s.t. for all $r<r_{x_0,x,y}$, $d^K_{a^2}(x,y)=d_{a^2}(x,y)-2a^2(x_0)r$,
\item If $x_0\in\p\o$ and $x_0$ is on a geodesic joining $x,y$ in $(\R^3,d_{a^2})$, then there is $r_{x_0,x,y}>0$ s.t. for all $r<r_{x_0,x,y}$, $d^K_{a^2}(x,y)=d_{a^2}(x,y)-(1+b^2)r$,
\item If $x_0$ is not on a geodesic joining $x,y$ in $(\R^3,d_{a^2})$, then there is $r_{x_0,x,y}>0$ s.t. for all $r<r_{x_0,x,y}$, $d^K_{a^2}(x,y)=d_{a^2}(x,y)$.
\end{enumerate}
\end{enumerate}
\end{prop}
Proposition \ref{P5.PropertiesdK} is proved in the \ref{S5.ProofP5.PropertiesdK}.

As a direct consequence of Proposition \ref{P5.PropertiesdK} we have the following.
\begin{remark}\label{R5.NouvelleRemark}
Let $\cl\subset\p\O$ as above and $x_0\in\O\setminus\cl$. 
\begin{enumerate}
\item If for all minimal connection $\sigma$ of $\cl$ and  for $i\in\{1,...,k\}$, we have that $x_0$ is not on a geodesic joining $p_i$ with $n_{\sigma(i)}$ in $(\R^3,d_{a^2})$, then there is $r_{x_0,\cl}>0$ s.t. for all $0<r<r_{x_0,\cl}$, we have for $K=\overline{B(x_0,r)}$
\begin{equation}\label{5.Someequaliiesforgeodesiccompactlink}
L(\cl,d^K_{a^2})=L(\cl,d_{a^2}).
\end{equation}
\item Assume that there exists a unique geodesic link $\cup\Gamma_i$ joining $\Pos$ with $\Neg$ (we may assume that $\sigma={\rm Id}$ is the minimal connection). Therefore from Remark \ref{R5.GeodLink}.1, for $i\neq j$we have $\Gamma_i\cap\Gamma_j=\emptyset$. If for some $i$ we have that $x_0$ which is on the geodesic joining $p_i$ with $n_i$, then there is  $r_{x_0,\cl}>0$ s.t. for all $0<r<r_{x_0,\cl}$, we have for $K=\overline{B(x_0,r)}$
\begin{equation}\label{5.Someequaliiesforgeodesiccompactlinkbisrepetita}
L(\cl,d^K_{a^2})=L(\cl,d_{a^2})-a^2\H^1_{|\cup\Gamma_i}(K)=\begin{cases}L(\cl,d_{a^2})-2a^2(x_0)r&\text{if }x_0\notin\p\o\\L(\cl,d_{a^2})-(1+b^2)r&\text{if }x_0\in\p\o\end{cases}.
\end{equation}
\end{enumerate}
\end{remark}
\subsubsection{The proofs of Propositions \ref{C5.StructureFunctionCompact} $\&$ \ref{C5.StructureFunctionCompactBisBis}}
This pseudo metric allows us to construct an intermediate structure function in the spirit of Proposition \ref{P5.1} which is constant in a "small" compact $K$.
\begin{prop}\label{P5.1compact}
Let  $K=\overline{B(x_0,r)}$ be s.t. $\overline{B(x_0,2r)}\subset\R^3\setminus\cl$ and $\eta>0$. 

Then there is $\delta_{\eta,K}>0$ s.t. for $0<\delta<\delta_{\eta,K}$ there are $C_{\eta,K,\delta}>0$, $E_{\eta,K,\delta}\subset\R$ and $\xi_{\eta,K,\delta}\in C^\infty(\R^3,\R)$ satisfying
\begin{enumerate}
\item $|\n\xi_{\eta,K,\delta}|\leq\alpha_\delta $ in $\R^3$ and $\xi_{\eta,K,\delta}$ is constant in $K$ ({\it i.e.} $|\n\xi_{\eta,K,\delta}|\leq\alpha_\delta\1_{\R^3\setminus K}$),
\item $\sum_{i\in\N_k}\{\xi_{\eta,K,\delta}(p_i)-\xi_{\eta,K,\delta}(n_i)\}\geq L(\cl,d_{\alpha_\delta}^K)-\eta$,
\item $\mathscr{H}^1(E_{\eta,K,\delta})\leq\eta$ and for $t\in\R\setminus E_{\eta,K,\delta}$, $\{\xi_{\eta,K,\delta}=t\}$ is a closed hypersurface whose second fundamental form is bounded by $C_{\eta,K,\delta}$.
\end{enumerate}
\end{prop}
\begin{proof}[Proof of Proposition \ref{C5.StructureFunctionCompact}]
Assume that $\sigma={\rm Id}$ is a minimal connection in $(\cl,d_{a^2})$. 

Let $r_{x_0,\cl}>0$ (given by Remark \ref{R5.NouvelleRemark}) be s.t. for  $K=\overline{B(x_0,r_{x_0,\cl}/2)}$, we have 
$L(\cl,d_{a^2})=L(\cl,d^K_{a^2})$. 
We fix $r_{x_0}$ of Proposition \ref{C5.StructureFunctionCompact} s.t.: $r_{x_0}\leq r_{x_0,\cl}/2$ and $\overline{B(x_0,2r_{x_0})}\cap\cl=\emptyset$. We let $K=\overline{B(x_0,r_{x_0})}$.

Now we apply Proposition \ref{P5.1compact}: there is $\delta_{\eta/2,K}>0$ s.t. for $0<\delta< \delta_{\eta/2,K}$, there are $C_{\eta/2,K,\delta}>0$, $E_{\eta/2,K,\delta}\subset\R$ and $\xi_{\eta/2,K,\delta}\in C^\infty(\R^3,\R)$ satisfying the conclusions of Proposition \ref{P5.1compact}.

From \eqref{5.ContinuityOfMinimalConnexionCompact} and  \eqref{5.Someequaliiesforgeodesiccompactlink}, one may fix $0<\delta< \delta_{\eta/2,K}$ s.t. 
\[
L(\cl,d_{a^2})- L(\cl,d^K_{\alpha_{\delta}})<\eta/2.
\]

We consider $\v_{\eta,K}>0$ s.t. for $0<\v<\v_{\eta,K}$, we have $C\e^{-\gamma \delta/\v}<\v^4\,(C\text{ and }\gamma$ are given by \eqref{5.Uepsprochea}). We obtain the result by taking $C_{\eta,K}=C_{\eta/2,K,\delta}$, $E_{\eta,K}=E_{\eta/2,K,\delta}$ and $\xi_{\eta,K}=\xi_{\eta/2,K,\delta}$.
\end{proof}
\begin{proof}[Proof of  Proposition \ref{C5.StructureFunctionCompactBisBis}]
Assume that $\sigma={\rm Id}$ is {\bf the} minimal connection of $\cl$ and that $x_0$ is on the geodesic joining $p_1$ with $n_1$. Let $r_{x_0}<\min\{ r_{x_0,p_1,n_1},\min_{i=2,...,k} r_{x_0,p_i,n_i}\}$ ($r_{x_0,p_i,n_i}$ given by Proposition \ref{P5.PropertiesdK}.2) s.t. $B(x_0,2r_{x_0})\subset\R^3\setminus\cl$.

It suffices to apply Proposition \ref{P5.1compact} as in the proof of Proposition \ref{C5.StructureFunctionCompact} combined with \eqref{5.Someequaliiesforgeodesiccompactlinkbisrepetita}. 
\end{proof}
\subsubsection{The proof of Proposition \ref{P5.1compact}}
The main point is that we require that $\xi_{\eta,K,\delta}$ is constant in $K$. All the other requirements are satisfied by the map $\xi_{\eta,\delta}$ constructed in Proposition \ref{P5.1}.

For $\delta<r/2$, let $K_1=\overline{B(x_0,r+2\delta)}$ and $K_2=\overline{B(x_0,r+\delta)}$. We denote $\alpha=\alpha_\delta$ and $\alpha'=\alpha_{2\delta}$.

We assume that $\sigma={\rm Id}$ is a minimal connection for $(\cl,d_{\alpha'}^{K_1})$ and we fix $\eta$.
\vspace{2mm}
\\
{\bf Step 1:} As in the proof of Proposition \ref{P5.1}, there is a function $\xi_0:\cl\to\R$ $1$-Lipschitz in $(\cl,d_{\alpha'}^{K_1})$ and s.t. $\xi_0(p_i)-\xi_0(n_i)=d_{\alpha'}^{K_1}(p_i,n_i)$.
\vspace{2mm}
\\{\bf Step 2:} We extend $\xi_0$ to a map $\xi_1:\R^3\to\R$, $1$-Lipschitz w.r.t. $d_{\alpha'}^{K_1}$ and thus constant in  $K_1$
\vspace{2mm}
\\For example, we may take
\[
\xi_1(x)=\max_i\{\xi_0(p_i)-d^{K_1}_{\alpha'}(x,p_i)\}.
\]
As in the proof of Proposition \ref{P5.1}, $\xi_{1|\cl}=\xi_0$ and $|\n \xi_1|\leq \alpha'$. Moreover, $\xi_1$ is constant in $K_1$. Indeed, for all $x\in K_1$, we have $\xi_0(x)=\max_i\xi_0(p_i)-d^{K_1}_{\alpha'}(x,p_i)=\max_i\xi_0(p_i)-d_{\alpha'}(K_1,p_i)$.
\vspace{2mm}
\\{\bf Step 3:}  For $0<\beta<\delta$, we approximate $\xi_1$ by $\xi_2\in C^\infty(\R^3,\R)$ satisfying $|\n\xi_2|\leq (1-\beta) \alpha$, $\sum_{i\in\N_k}\{\xi_2(p_i)-\xi_2(n_i)\}\geq L(\cl,d^K_\alpha)-\eta/2$ (for $\delta$ sufficiently small), and s.t. $\xi_2$ is constant in $K_2$\vspace{2mm}
\\The approximation $\xi_2$ is obtained (as in Proposition \ref{P5.1}) by regularization of $(1-\beta)\xi_1$ using a mollifier $\rho_t$ ($t<10^{-2}\delta$) and by noting that
\[
L(\cl,d^K_{\alpha})\geq L(\cl,d^K_{\alpha'})\geq L(\cl,d^{K_2}_{\alpha'})\geq L(\cl,d^{K_1}_{\alpha'})\geq L(\cl,d^K_{\alpha'})-\mathcal{O}(\delta)= L(\cl,d^K_{\alpha})-\mathcal{O}(\delta).
\]
We fix $\delta_{\eta,K}>0$ s.t. for $\delta<\delta_{\eta,K}$ we have 
\begin{enumerate}[$\bullet$]
\item $L(\cl,d^{K_1}_{\alpha'})\geq L(\cl,d^K_{\alpha})-\eta/8$ and $2k\eta^2\delta^2\leq 10^{-2}\eta$ (this condition is used below)
\item $\|\xi_1-\xi_2\|_{L^\infty}\leq{\eta}/({16k})$ and $\delta L(\cl,d_{a^2})\leq\eta/4$.
\end{enumerate}
\vspace{2mm}
{\bf Step 4:} Let $\tilde{\O}$ be a neighborhood of $\overline{\O}$. We approximate $\xi_2$ in $C^1(\tilde{\O})$ by a Morse function $\xi_3\in C^\infty(\R^3,\R)$ 

We let $\xi_3\in C^\infty(\R^3,\R)$ be s.t.
\[
\|\xi_3-\xi_2\|_{C^1(\tilde{\O})}<\eta^2\delta^2,
\]
\[
|\n\xi_3|\leq(1-\beta/2) \alpha,
\]
\[
\xi_3\text{ is a Morse function},
\]
\[
\exists R>0\text{ s.t. in }\R^3\setminus B(0,R),\,\xi_3=|x|/2.
\]
\vspace{2mm}
\\{\bf Step 5:} We modify $\xi_3$ in order to have $\xi_{\eta,K,\delta}\equiv C_0$ in $K$
\vspace{2mm}
\\
By construction, there is $C_0\in\xi_3(K)$ s.t. $\|\xi_3-C_0\|_{C^1(K_2)}<\eta^2\delta^2$. Noting that $\dist(\p K_2,K)=\delta$, one may construct $\xi_{\eta,K,\delta}\in C^\infty(\R^3)$ s.t.
\[
\left\{\begin{array}{c}\xi_{\eta,K,\delta}=\xi_3\text{ in }\R^3\setminus K_2,\,\xi_{\eta,K,\delta}\equiv C_0\text{ in }K,\\\|\xi_{\eta,K,\delta}-C_0\|_{L^\infty(K_2)}<\eta^2\delta^2 \text{ and }|\n\xi_{\eta,K,\delta}|\leq b^2\text{ in }K_2.\end{array}\right.
\]
Clearly $\xi_{\eta,K,\delta}$ satisfies {\it 1.} and  {\it 2.} in Proposition \ref{P5.1compact}. 
\vspace{2mm}
\\{\bf Step 6:} We construct $E_{\eta,K,\delta}$ 
\vspace{2mm}

For $\rho>0$, we consider $E_{\eta,K,\delta}^1=\xi_{\eta,K,\delta}(\cup_iB(x_i,\rho))$ where $\{x_1,...,x_l\}$ is the set of the critical points of $\xi_{\eta,K,\delta}$ in $B(0,R)\setminus K_2$. 

For the same reasons as in Proposition \ref{P5.1}, we have $\mathscr{H}^1(E_{\eta,K,\delta}^1)\leq C\rho$.

We also define $E_{\eta,K,\delta}^2=\xi_{\eta,K,\delta}(K_2)$. By construction, we have $\mathscr{H}^1(E_{\eta,K,\delta}^2)\leq 2\eta^2\delta^2$.

Thus it suffices to consider $\rho$ sufficiently small in order to have  $ C\rho\leq10^{-2}\eta$ and to set $E_{\eta,K,\delta}=E_{\eta,K,\delta}^1\cup E_{\eta,K,\delta}^2$.

\subsection[Second step in the proof of Theorem \ref{T5.MainSymCase}]{Second step in the proof of Theorem \ref{T5.MainSymCase}: a structure function in presence of symmetries}\label{S5.SymCase}
In this section we assume that $\O=B(0,1)$ and that $\o=B(0,r_0)$, with $r_0\in(0,1)$. 

Consider $\cl=\{(1,0,0),(-1,0,0)\}=\{p,n\}$, $p=(1,0,0)$. It is clear that in this situation, the line segment $[p,n]$ is the unique geodesic between $p$ and $n$ in $(\R^3,d_{a^2})$.  

In order to get a very sharp lower bound (matching with the upper bound up to a $\mathcal{O}(1)$ term), we need a structure function as in Proposition \ref{C5.StructureFunctionCompact} but with "$\eta=0$".

This structure function is given in the following proposition.
\begin{prop}\label{P5.StructFonctSymCas}
Let $M\in\O\setminus[p,n]$. Then there is   $\mathcal{V}$,  an open neighborhood of $M$ s.t. for $\v>0$, there is $\xi_\v:\R^3\to\R$ a Lipschitz function s.t.
\begin{enumerate}
\item $\xi_\v(p)-\xi_\v(n)=d_{U_\v^2}(p,n)$,
\item $|\n \xi_\v|\leq U^2_\v$,
\item $\xi_\v\equiv0$ in $\mathcal{V}$,
\item $\forall\, t\in\xi_\v(\R^3)\setminus\{0,\xi_\v(p),\xi_\v(n)\}$, $\{\xi_\v=t\}$  is a sphere whose radius is at least $1$.
\end{enumerate}
\end{prop}
The proof of Proposition \ref{P5.StructFonctSymCas} is in the next section (Section \ref{SPreuveTroisStrucFonqsdjfh}).
\subsection{Proof of Proposition \ref{P5.StructFonctSymCas}}\label{SPreuveTroisStrucFonqsdjfh}

Using the spherical symmetry of $\O$, $\o$ and the minimality of $U_\v$, one may easily prove the following proposition.
\begin{prop}\label{P5.DomaineSymetrique}
The unique minimizer $U_\v$ of $E_\v$ in $H^1_1$, is radially symmetric and non decreasing.
\end{prop}

Proposition \ref{P5.StructFonctSymCas} is a particular case of the following lemma (by taking $U=U_\v$).

\begin{lem}\label{L5.Lemmedelaltere}[The dumbbell lemma]

Let $U:\R^3\to[b,1]$ be a radially symmetric and non decreasing Borel function. Fix $p,n\in \S^2$, $p=-n$ and let $M\in \O\setminus[p,n]$.

Then there are $\xi:\R^3\to\R$ and $B^+$, $B^-$ two distinct open balls, $B^+,B^-$ are exteriorly tangent and independent of $U$ s.t.
\begin{enumerate}
\item $\xi(p)-\xi(n)=d_{U^2}(p,n)$,
\item $|\n \xi| \leq U^2$,
\item $\xi\equiv0$ in $\mathcal{V}:=\R^3\setminus (B^+\cup B^-)$,
\item $M\in T$ with $T$ which is the common tangent plan of $B^+$ and $B^-$,
\item $B^+$ is centered in $2p$, $B^-$ is centered in $2n$,
\item denoting $\tilde{B}^+$ (resp. $\tilde{B}^-$) the ball centered in $2p$ (resp. $2n$) with radius $1$, $\xi$ is locally constant in  $\tilde{B}^+\cup\tilde{B}^-$,
\item $\forall t\in\xi(\R^3)\setminus\{0,\xi(p),\xi(n)\}$, $\{\xi=t\}$ is a sphere centered in $2p$ or $2n$ whose radius is at least $1$.
\end{enumerate}
\end{lem}
Using the symmetry of the situation, the function $\xi$ is represented in the Figure \ref{F5.GeomOfXi}. 
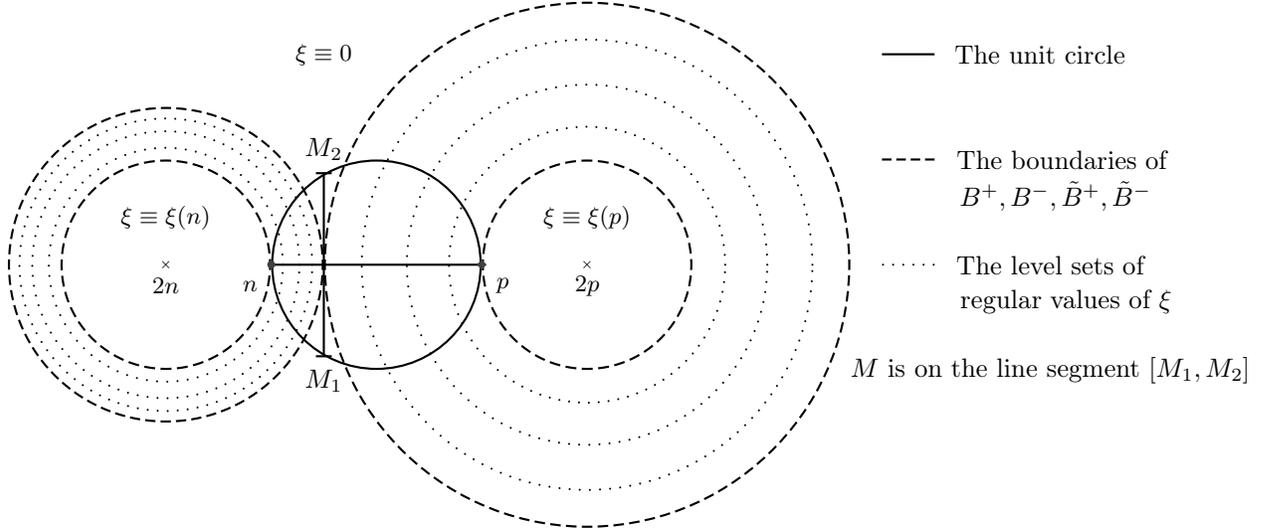
\begin{figure}[h]
\psset{xunit=1.4cm,yunit=1.4cm,algebraic=true,dotstyle=o,dotsize=3pt 0,linewidth=0.8pt,arrowsize=3pt 2,arrowinset=0.25}
\begin{pspicture*}(-3.511,-2.7)(10,2.7)
\pscircle(0,0){1.4}
\pscircle[linestyle=dashed,dash=4pt 2pt](-2,0){1.4}
\pscircle[linestyle=dashed,dash=4pt 2pt](2,0){1.4}
\psline{|-|}(-0.5,-.88)(-0.5,.88)
\rput(-.5,-1.1){$M_1$}
\rput(-.5,1.1){$M_2$}
\pscircle[linestyle=dashed,dash=4pt 2pt](-2,0){2.1}
\pscircle[linestyle=dashed,dash=4pt 2pt](2,0){3.5}
\pscircle[linestyle=dotted](2,0){3.01}
\pscircle[linestyle=dotted](2,0){2.41}
\pscircle[linestyle=dotted](2,0){1.85}
\pscircle[linestyle=dotted](-2,0){1.57}
\pscircle[linestyle=dotted](-2,0){1.96}
\pscircle[linestyle=dotted](-2,0){1.79}
\psline(-1,0)(1,0)
\psdots[dotstyle=x](-2,0)\rput(-2,-.2){\small$2n$}
\psdots[dotstyle=x](2,0)\rput(2,-.2){\small$2p$}
\psdots[dotstyle=*,linecolor=darkgray](-1,0)\rput(-1.2,-.2){\small$n$}
\psdots[dotstyle=square*,dotangle=45,linecolor=darkgray](1,0)\rput(1.2,-.2){\small$p$}
\rput(-0.5,2){\small$\xi\equiv0$}
\rput(-2.,.45){\small$\xi\equiv\xi(n)$}
\rput(2.,.45){\small$\xi\equiv\xi(p)$}
\psline(4.8,2)(5.3,2)\rput(6.3,2){The unit circle}
\psline[linestyle=dashed,dash=4pt 2pt](4.8,1)(5.3,1)\rput(6.55,1){The boundaries of }\rput(6.45,0.65){$B^+,B^-,\tilde{B}^+,\tilde{B}^-$}
\psline[linestyle=dotted](4.8,0)(5.3,0)\rput(6.4,0){The level sets of}\rput(6.55,-.35){regular values of $\xi$}
\rput(6.4,-1){$M$ is on the line segment $[M_1,M_2]$}
\end{pspicture*}\caption[The geometry of the level sets of $\xi$]{The geometry of the level sets of $\xi$ (intersected with the plane defined by $p,n,M$)}\label{F5.GeomOfXi}
\end{figure}
\begin{proof} Let $p,n\in\p \O$, $p=-n$ and $\{0,(e_1,e_2,e_3)\}$ be an orthonormal and a direct coordinate system of $\R^3$ s.t. $p=(1,0,0)$ and $n=(-1,0,0)$. Let $M(x_0,y_0,z_0)\in\O\setminus[p,n]$.\\
{\bf Step 1:}  $\xi_0:[-1,1]\to\R$ s.t. $\xi_0(1)-\xi_0(-1)=d_{U^2}(p,n)$, ${\xi_0}^\prime(s)=U^2(s,0,0)$ and $\xi_0 (x_0)=0$

It suffices to consider $\displaystyle {\xi_0}(s)=\int_{x_0}^s{U^2(t,0,0)\di t}$.
\\{\bf Step 2:} We construct $\xi:\R^3\to\R$

We denote
\[
\Sigma_r^+=\p B((2,0,0),r)\text{ for $r\in(1,2-x_0)$}
\]
and
\[
\Sigma_r^-=\p B((-2,0,0),r)\text{ for $r\in(1,2+x_0)$}.
\]
We define $\xi:\R^3\to\R$ by its level sets:
\[
\xi=
\begin{cases}
\xi_0(2-r)&\text{on }\Sigma_r^+, \,r\in(1,2-x_0)\\\xi_0(r-2)&\text{on }\Sigma_r^-,\,r\in(1,2+x_0)\\\xi_0(-1)&\text{in }\overline{B((-2,0,0),1)}\\\xi_0(1)&\text{in }\overline{B((2,0,0),1)}\\0&\text{otherwise}
\end{cases}.
\]
{\bf Step 3: } $\xi$ satisfies the properties of Lemma \ref{L5.Lemmedelaltere}

Assertion {\it1.} is easily satisfied since $\xi(p)=\xi_0(1)$, $\xi(n)=\xi_0(-1)$ and $\xi_0(1)-\xi_0(-1)=d_{U^2}(p,n)$.

We take $B^+=B((2,0,0),2-x_0)$ and $B^-=B((-2,0,0),2+x_0)$. 

Clearly Assertions {\it3., 4., 5., 6.} and {\it 7.} hold.

We check  {\it2.}.  Since $\xi$ is locally constant in $V:=[\R^3\setminus (B^+\cup B^-)]\cup \overline{\tilde{B}^+\cup\tilde{B}^-}$, it suffices to prove that $|\n\xi|\leq U^2$ in $\R^3\setminus V$. 

The key argument is the fact that for $Q,Q'\in\R^3$, $Q\neq Q'$ and $0<r<|Q-Q'|$ we have $\dist(Q,\p B(Q',r))=|Q-Q'|-r=|Q-Q_0|$ where $[Q,Q']\cap\p  B(Q',r)=\{Q_0\}$. This is obvious if we draw a picture and may be easily justified. Indeed, if $Q_0$ is a minimal point, then the line segment $[Q,Q_0]$ is orthogonal to $\p B(Q',r)$. Only two points on $\p B(Q',r)$ satisfy this condition and  one of them is clearly not minimal.

Consequently, by  taking $Q=0$ and $Q'\in\{2p,2n\}$ we have that 
\[
\min_{Q_0\in \Sigma_r^\pm}|Q_0|=|(\pm(2-r),0,0)|.
\]

Note that $U$ is radially symmetric and non decreasing. Since in each connected components of
\[
(B^+\cup B^-)\setminus \overline{\tilde{B}^+\cup\tilde{B}^-},
\] 
$\xi$ admits a spherical symmetry, we have 
\begin{eqnarray*}
|\n \xi(x)|&=&\begin{cases}
|{\xi_0}'(2-r)|=U^2(2-r,0,0)=\min_{\Sigma^+_{r}}U^2&\text{if }x\in\Sigma_r^+\\|{\xi_0}'(r-2)|=U^2(r-2,0,0)=\min_{\Sigma^-_{r}}U^2&\text{if }x\in\Sigma_r^-
\end{cases}
\\&\leq& U^2(x).
\end{eqnarray*}
\end{proof}

\section{Lower bound for ${\rm inf}_{H^1_g}F_\v$ when $g\in\mathcal{H}$: the argument of Sandier}\label{S5.LowerBoundSandier}
As briefly explain in Section \ref{S5.StructureFunction}, we compute a lower bound { via} the Coarea formula: we integrate on the level sets of $\xi:\R^3\to\R$ (a structure function).

Thus the obtention of a lower bound for $F_\v(v_\v)$ is related with lower bounds on hypersurfaces. On hypersurfaces, the main ingredient to get the desired estimate is Proposition 3.5 in \cite{Sandier1}. For the convenience of the reader, we recall this result.
\begin{prop}\label{P5.lemmedeSandier}
Let $\tilde{\Sigma}$ be a (smooth) closed and oriented hypersurface in $\R^3$ whose second fundamental form is bounded by $K$. We denote by $d(\cdot,\cdot)$ the Euclidean distance restricted to $\tilde{\Sigma}$.

Consider $\Sigma\subset\tilde{\Sigma}$, a (smooth) bounded open set and $v:\Sigma\to\C$ s.t. there is  $0<\beta<1$ satisfying
\[
\dist(x,\p\Sigma)<\beta\Rightarrow|v(x)|\geq1/2.
\]
Then we have the existence of  $C>0$ depending only on  $K$ and $\deg(v,\p\Sigma)$ s.t.
\[
\frac{1}{2}\int_{\Sigma}{\left\{|\n v|^2+\frac{1}{2\v^2}(1-|v|^2)^2\right\}}\geq \pi|\deg(v,\p\Sigma)|\ln\frac{\beta}{\v}-C.
\]
\end{prop}
This section is devoted to the proof of the following proposition.
\begin{prop}\label{P5.BorneInfLeCasGénéral}
Let $g\in\mathcal{H}$ and let $\cl=\Pos\cup \Neg$ be the set of its singularities. 
\begin{enumerate}[1)]
\item We have
\begin{equation}\label{5.BorneLimInfBis}
\liminf_{\v\to0}{\frac{F_{\v}(v_\v)}{|\ln\v|}}\geq\pi L(g,d_{a^2}).
\end{equation}
\item We denote by $<\Gamma>$ the union of all geodesic links of $\cl$ in $(\R^3,d_{a^2})$ and for $\mu>0$, $K_\mu:=\{x\in\O\,|\,\dist(x,<\Gamma>)\geq\mu\}$. Then we have 
\begin{equation}\label{5.BorneLimInfBisCompact}
\liminf_{\v\to0}{\frac{F_{\v}(v_\v,\O\setminus K_\mu)}{|\ln\v|}}\geq\pi L(g,d_{a^2}).
\end{equation}
\item Assume that there exists a unique geodesic link $\cup\Gamma_i$ joining the singularities of $g$. Then for $x_0\in\cup_i\Gamma_i\setminus\cl$ there exists $r_{x_0}>0$ s.t. letting $K=\overline{B(x_0,r_{x_0})}$ we have
\begin{equation}\label{5.BorneLimInfBisCompactAlelkqsjdf}
\liminf_{\v\to0}{\frac{F_{\v}(v_\v,\O\setminus K)}{|\ln\v|}}\geq\pi \left[L(g,d_{a^2})-a^2\H^1_{|\cup\Gamma_i}(K)\right].
\end{equation}

\item Moreover, if we are in the symmetric case of Section \ref{S5.SymCase}, then we have $<\Gamma>=[p,n]$ and  there is $C_\mu>0$ s.t.
\begin{equation}\label{5.BorneLimInfSymCase}
F_{\v}(v_\v,\O\setminus K_\mu)\geq\pi d_{a^2}(p,n)|\ln\v|-C_\mu.
\end{equation}
\end{enumerate}
\end{prop}
Theorem \ref{T5.Main1} for $g\in\mathcal{H}$, as well as Theorems  \ref{T5.Main2}$\&$\ref{T5.MainSymCase}, are straightforward consequences of Proposition \ref{P5.BorneInfLeCasGénéral} combined with the upper bounds \eqref{5.UpperBoundGeodesicareLines}, \eqref{5.UpperBoundGeneralCase}.

We prove in detail \eqref{5.BorneLimInfBis}, and we will sketch the proofs of \eqref{5.BorneLimInfBisCompact},\eqref{5.BorneLimInfSymCase} which are, as explained in \cite{Sandier1}, obtained exactly in the same way as \eqref{5.BorneLimInfBis}.

We prove that for all $\tilde{\eta}:=\eta(8k^2+3k+1)>0$, the following holds
\begin{equation}\label{5.BorneLimInf}
\liminf_{\v\to0}{\frac{F_{\v}(v_\v)}{|\ln\v|}}\geq\pi L(g,d_{a^2})-\tilde{\eta}.
\end{equation}
Let $\eta>0$, $\v_n\downarrow0$, let $(v_n)_n\subset H^1_g$ be a sequence of minimizers of $F_{\v_n}$ in $H^1_g$ and let $\xi_\eta,C_\eta,E_\eta$ be given by Proposition \ref{C5.Ustruc} (for $n$ sufficiently large). 


Let $0<\rho<10^{-2}\min\{\eta,\v_\eta\}$ ($\v_\eta$ defined in Proposition \ref{C5.Ustruc}) and set
\[
\O_\rho:=\{x\in\R^3\,|\,\dist(x,\O)<\rho\text{ and }\dist(x,\cl)>\rho\}.
\]
One may assume that $\rho$ is sufficiently small s.t. in $\O_\rho\setminus\O$, $\Pi_{\p\O}$, the orthogonal projection on $\p\O$, is well defined and smooth.

Then we extend $v_n$ (we use the same notation for the extension) by letting 
\[
v_n:\O_\rho\to\R^2,\,x\mapsto\begin{cases}v_n(x)&\text{if }x\in\O\\g(\Pi_{\p\O}(x))&\text{if }x\in\O_\rho\setminus\O\end{cases}.
\] 
Since $g\in\mathcal{H}$ and $v_{n|\O_\rho\setminus\O}$ does not depend on $n$ and takes its values in $\S^1$, we obtain the existence of $C_0(\rho)$ depending only on $\rho,\O,g$ s.t. (for small $\rho$)
\begin{equation}\label{LowerBoundOutsideBis}
F_{\v_n}(v_n,\O)\geq F_{\v_n}(v_n,\O_\rho)-C_0(\rho)
\end{equation}
If we define $F=F_{\eta,\rho}:=E_\eta\cup[\xi_\eta(\cl)-2\rho,\xi_\eta(\cl)+2\rho]$ ($\xi_\eta,E_\eta$ given by Proposition \ref{C5.Ustruc}), then we have 
\[
\mathscr{H}^1(F)\leq8k\rho+\eta\leq(8k+1)\eta.
\]

If $t\in\R\setminus F$, we denote $\tilde{\Sigma}_t:=\{\xi_\eta=t\}$. We construct for almost all $t\in\R\setminus F$ a smooth closed submanifold $\Sigma_t\subset \tilde{\Sigma}_t$. 

Note that for $t\in\R\setminus F$, we have $\dist(t,\xi_\eta(\cl))\geq2\rho$ (by definition of $F$). Consequently, for $t\in\R\setminus F$, we obtain that $\tilde\Sigma_t\cap \left\{\O+B(0,\rho)\right\}=\tilde\Sigma_t\cap \O_\rho$ (because $\xi_\eta$ is $1$-Lipschitz).


Since $t\in\R\setminus F$ is not a critical value of $\xi_\eta$, the connected components $W$'s of $\tilde\Sigma_t=\p\{\xi_\eta\geq t\}=\{\xi_\eta=t\}$ have no boundary. If such $W$ intersects $\O_\rho$, then we distinguish two cases:
\begin{enumerate}[a)]
\item $W\cap\p\O_\rho=\emptyset$
\item $W\cap\p\O_\rho\neq \emptyset$.
\end{enumerate}
Denote by $W_a$, resp. $W_b$, the set of the connected components satisfying a), resp. b).

If $W_b=\emptyset$, then we define $\Sigma_t=\tilde\Sigma_t\cap \O_\rho=\{\xi_\eta=t\}\cap \O_\rho$. 

Thus it remains to construct $\Sigma_t$ when $W_b\neq\emptyset$. Consider
\[
\begin{array}{cccc}
f:&\O+B(0,\rho)&\to&\R^2\\&x&\mapsto&\left(\xi_\eta(x),\dist[x,\p(\O+B(0,\rho))]\right)
\end{array}.
\]

Using the Constant Rank Theorem (see Theorem 4.3.2, page 91 in \cite{Demazure1}), the set $f^{-1}(\{t\}\times[r,\infty))$ ($r\in(0,\rho/2)$) is a manifold with boundary when
\begin{enumerate}[$\bullet$]
\item $t$ is a regular value of $\xi_\eta$,
\item $(t,r)$ is a regular value of $f$.
\end{enumerate}
Thus, using Sard's Lemma, for almost all $t\in\R\setminus F$ s.t. $W_b\neq\emptyset$, there is $r=r(t)\in(0,\rho/2)$ s.t. 
\begin{equation}\label{Equa.r(t)}
\text{$\Sigma_t=f^{-1}(\{t\}\times[r,\infty))\subset\tilde{\Sigma}_t$ is a closed submanifold with boundary.}
\end{equation}
Moreover, we have $\p\Sigma_t\subset \p\{\O+B(0,\rho-r)\}\cap\O_\rho$.

We denote by $G$ the set 
\[
G:=\{t\in\R\setminus F\,|\,W_b=\emptyset\text{ or }W_b\neq\emptyset\text{ and there is $r(t)\in(0,\rho/2)$ s.t. \eqref{Equa.r(t)} holds}\}.
\] 

For $t\in G$ we have
\begin{equation}\label{L5.BordHypersurfaceLoin}
\dist(\p\Sigma_t,\O)\geq\rho/2.
\end{equation}





Let $x\in\Sigma_t$ be s.t. $\dist(x,\p\Sigma_t)<\rho/2$. Using \eqref{L5.BordHypersurfaceLoin}, we have $x\in\O_\rho\setminus\O$ and therefore $|v_n(x)|=1$. 

Finally, we are in a position to apply Proposition \ref{P5.lemmedeSandier}: 
\begin{equation}\label{5.ApplicationLemmeSandier}
\frac{1}{2}\int_{\Sigma_t}{\left\{|\n v_n|^2+\frac{b^2}{2\v^2_n}(1-|v_n|^2)^2\right\}}\geq\pi|\deg(v_n,\p\Sigma_t)|\ln\frac{b\rho}{\v_n}-C(\deg(v_n,\p\Sigma_t)).
\end{equation}
For $M\in\cl$ and for $t\in G$ we denote $M^{t}\in\p(\O+B(\rho-r(t))$ s.t. $\Pi_{\p\O}(M^t)=M$. Here we set $r(t)=0$ when $W_b=\emptyset$, \emph{i.e.}, when $\Sigma_t=\tilde{\Sigma}_t\cap\O_\rho$. It is clear that $M^t$ is uniquely defined.

Since $d(n,t)=\deg(v_n,\p\Sigma_t)={\rm Card}(\{p_i^t\in\{\xi_\eta\geq t\}\})-\Card(\{n_i^t\in\{\xi_\eta\geq t\}\})$ takes at most $2k$ values, one may assume that $C(\deg(v_n,\p\Sigma_t))$ is uniformly bounded in $n$ and  $t$. Note that $d(n,t)$ is defined for almost all $t\in \R\setminus F$ (for $t\in G$).

A key argument in this proof is the way to pass from lower bounds on hypersurfaces to a lower bound in $\O$. To do this we have the following.

\begin{lem}\label{L5.intdeg}
The following lower bound holds
\[
\int_{\R\setminus F}{d(n,t)\,{\rm d}t}\geq L(g,d_{a^2})-\eta(8k^2+3k+1).
\]
\end{lem}
\begin{proof}
Let $m=\inf_{\O_\rho}\xi_\eta$ and for $t\in\R$ let 
\begin{eqnarray*}
S(t)&=&{\rm Card}\left\{p_i\in\{\xi_\eta\geq t+\rho\}\right\}-{\rm Card}\left\{n_i\in\{\xi_\eta\geq t-\rho\}\right\}
\\&=&\sum_{i=1}^k\{\1_{\xi_\eta(p_i)\geq t+\rho}-\1_{\xi_\eta(n_i)\geq t-\rho}\}.
\end{eqnarray*}
We have
\begin{eqnarray*}
\int_{\R\setminus F}{d(n,t)\,{\rm d}t}&=&\int_{\R\setminus F}{{\rm Card}\{p_i^t\in\{\xi_\eta\geq t\}\}-{\rm Card}\{n_i^t\in\{\xi_\eta\geq t\}\}}
\\\text{[$\xi_\eta$ is $1$-Lipschitz]}&\geq&\int_{\R\setminus F}{{\rm Card}\{p_i\in\{\xi_\eta\geq t+\rho\}\}-{\rm Card}\{n_i\in\{\xi_\eta\geq t-\rho\}\}}
\\ \text{[$S(t)\leq k\,\&\,\mathscr{H}^1(F)\leq(8k+1)\eta$]} &\geq&\int_{\R}{S(t)}-k(8k+1)\eta
\\&\geq&\sum_{i=1}^k\int_{m}^\infty\{\1_{\xi_\eta(p_i)\geq t+\rho}-\1_{\xi_\eta(n_i)\geq t-\rho}\}-k(8k+1)\eta
\\&\geq&\sum_{i=1}^k\left\{\xi_\eta(p_i)-\xi_\eta(n_i)\right\}-2k\rho-k(8k+1)\eta
\\\text{[$\rho\leq\eta$]}&\geq& L(g,d_{a^2})-(8k^2+3k+1)\eta.
\end{eqnarray*}
\end{proof}
With the help of Lemma \ref{L5.intdeg}, we have
\begin{align*}
F_{\v_n}(v_n,\O_\rho)\geq(\text{Prop. \ref{C5.Ustruc} $\&$ \eqref{LowerBoundOutsideBis}})\geq\:&\:\frac{1}{2}\int_\O{(|\n\xi_\eta|-\v_n^4)\left[|\n v_n|^2+\frac{U^2_{\v_n}}{2\v_n^2}(1-|v_n|^2)^2\right]}-C_0
\\\geq\text{\eqref{5.UpperBoundGeneralCase}}\geq\:&\:\int_{\R\setminus F}{\frac{1}{2}\int_{\Sigma_t}{\left\{|\n v_n|^2+\frac{b^2}{2{\v}^2_n}(1-|v_n|^2)^2\right\}}\di t}-(C_0+1)
\\\geq\text{\eqref{5.ApplicationLemmeSandier}}\geq\:&\:\pi(\ln\frac{b\rho}{\v_n}-C)\int_{\R\setminus F}{|d(n,t)|}-(C_0+1)
\\\geq\text{(Lemma \ref{L5.intdeg})}\geq\:&\:\pi(\ln\frac{b\rho}{\v_n}-C)\left[L(g,d_{a^2})-(8k^2+3k+1)\eta\right]-(C_0+1)
\\\geq\:&\:\pi|\ln\v_n|\left[L(g,d_{a^2})-(8k^2+3k+1)\eta\right]-\tilde{C}.
\end{align*}
It follows that
\[
\liminf_n\frac{F_{\v_n}(v_n,\O)}{|\ln\v_n|}\geq\pi L(g,d_{a^2})-(8k^2+3k+1)\eta,\,\forall\,\eta>0.
\]
Estimate \eqref{5.BorneLimInfBis} in Proposition \ref{P5.BorneInfLeCasGénéral} is obtained by letting $\eta\to0$ in the above estimate.

We now sketch the arguments leading to \eqref{5.BorneLimInfBisCompact} and \eqref{5.BorneLimInfSymCase}. The fundamental ingredient is a lower bound for $F_\v(v_\v,\O\setminus K_\mu)$. Without loss of generality, by compactness of $K_\mu$, we may only consider the situation $K=\overline{B(x,r_x)}$ for some $x$ which does not belong to a geodesic link between the singularities of $g$; here, $r_x>0$ is some small number.



In order to prove \eqref{5.BorneLimInfBisCompact}, we use Proposition \ref{C5.StructureFunctionCompact}: we consider $\xi_{\eta,K}$ s.t. $\n\xi_{\eta,K}=0$ in $K$.

Following the same lines of proof of \eqref{5.BorneLimInfBis}, we find that
\[
\liminf_n\frac{F_{\v_n}(v_n,\O\setminus K)}{|\ln\v_n|}\geq\pi L(g,d_{a^2}).
\]
Combining this lower bound with \eqref{5.UpperBoundGeneralCase}, we obtain
\begin{equation}\label{5.ConcentrationEnergySmallComp}
F_{\v_n}(v_n,K)=o(|\ln\v_n|).
\end{equation}
To get \eqref{5.BorneLimInfBisCompactAlelkqsjdf}, we use Proposition \ref{C5.StructureFunctionCompactBisBis} and we get
\[
\liminf_n\frac{F_{\v_n}(v_n,\O\setminus K)}{|\ln\v_n|}\geq\pi\left[ L(g,d_{a^2})-a^2\H^1_{|\cup\Gamma_i}(K)\right].
\]
This lower bound is exactly \eqref{5.BorneLimInfBisCompactAlelkqsjdf} and by combining  \eqref{5.BorneLimInfBisCompactAlelkqsjdf} with \eqref{5.UpperBoundGeneralCase} yields \eqref{mpfkjfjfjhfngngng}. Therefore, as explained in Section \ref{S5.StrFonCCompAAAA}, from \eqref{5.ConcentrationEnergySmallComp} and \eqref{mpfkjfjfjhfngngng} we prove Theorem \ref{T5.Main2}. 

In the symmetric case, using Proposition \ref{P5.StructFonctSymCas}, for $x\in\O\setminus[p,n]$ we obtain the existence of $r_x>0$ s.t., with $K=\overline{B(x,r_x)}$, we have
\[
F_{\v_n}(v_n,\O\setminus K)\geq\pi d_{a^2}(p,n)|\ln\v|-C_K.
\]
This estimate is obtained exactly as \eqref{5.BorneLimInfBis} because we may have a structure function $\xi$ with "$\eta=0$".

Consequently from the upper bound \eqref{5.UpperBoundGeodesicareLines}, we deduce
\[
F_{\v_n}(v_n, K)\leq C_K'
\]
which implies Theorem \ref{T5.MainSymCase}.

\section{Extension by density of Theorem \ref{T5.Main1}}\label{S5.DensityArgument}

From \eqref{5.BorneLimInfBisCompact} and \eqref{5.UpperBoundGeneralCase}, we obtain that Theorem  \ref {T5.Main1} holds for $g\in\mathcal{H}$. This section is devoted to the extension of Theorem \ref{T5.Main1} to the general case $g\in H^{1/2}(\p\O,\S^1)$.

For $g\in H^{1/2}(\p\O,\S^1)$, we denote
\[
f_{\v,g}=\min_{v\in H^1_g}F_\v(v).
\]

Using same arguments as in \cite{BBM1}, we have
\begin{prop}\label{P5.FundPropDenste1}
\begin{enumerate}
\item Let $\delta\in(0,1)$. There is $C(\delta)>0$ s.t. for $g_1,g_2\in H^{1/2}(\p\O,\S^1)$, we have ($\text{(5.1),(5.2) in \cite{BBM1}}$)
\begin{equation}\label{5.ComparaisonEnergyProd}
(1-\delta)f_{\v,g_1}-C(\delta)f_{\v,g_2}\leq f_{\v,g_1g_2}\leq (1+\delta)f_{\v,g_1}+C(\delta)f_{\v,g_2}.
\end{equation} 
\item There is $C>0$ depending only on $\O$ s.t. for $g\in H^{1/2}(\p\O,\S^1)$ we have $(\text{(5.4) in \cite{BBM1}})$
\begin{equation}\label{5.BorneGrossiereEnergy}
f_{\v,g}\leq C|g|^2_{H^{1/2}(\p\O)}(1+|\ln\v|).
\end{equation}
\item If $(g_n)_n\subset\mathcal{H}$ is s.t. $g_n\to g$  in $H^{1/2}(\p\O)$ then Lemma 17 in \cite{BBM1} applied with $u_n=g_n/g$ and $v=g$ yields 
\begin{equation}\label{5.ConvQuotient}
\left|\frac{g_n}{g}\right|_{H^{1/2}(\p\O)}\to0.
\end{equation}
\item There is $C>0$ depending only on $\O$ and on $a$ s.t. for $g_1,g_2\in H^{1/2}(\p\O,\S^1)$ we have ((2.6) in \cite{BBM1})
\begin{equation}\label{5.ContLongMin}
|L(g_1,d_{a^2})-L(g_2,d_{a^2})|\leq C|g_1-g_2|_{H^{1/2}(\p\O)}\left(|g_1|_{H^{1/2}(\p\O)}+|g_2|_{H^{1/2}(\p\O)}\right).
\end{equation}
\end{enumerate}
\end{prop}
\begin{proof}
The estimates \eqref{5.ComparaisonEnergyProd} $\&$ \eqref{5.BorneGrossiereEnergy} are obtained using exactly the same arguments than in \cite{BBM1}. Estimate \eqref{5.BorneGrossiereEnergy} is proved in \cite{BBM1}. We now focus on \eqref{5.ContLongMin} which is not exactly proved in \cite{BBM1} (similar estimates are obtained for other metrics). We have easily proved by the definition of $L(g,d_{a^2})=(2\pi)^{-1}\sup\left\{T_g(\varphi)\,\left|\,|\varphi|_{d_{a^2}}\leq1\right.\right\}$ and Proposition \ref{P5.BBM1AuxProp}.1 that
\begin{eqnarray*}
|L(g_1,d_{a^2})-L(g_2,d_{a^2})|&\leq& L({g_1}\overline{g_2},d_{a^2})\leq L({g_1}\overline{g_2},d_{\rm eucl})\\&\leq[\text{Eq. (2.8) in \cite{BBM1}}]\leq& C|g_1-g_2|_{H^{1/2}(\p\O)}\left(|g_1|_{H^{1/2}(\p\O)}+|g_2|_{H^{1/2}(\p\O)}\right).
\end{eqnarray*}
\end{proof}

Using this proposition, Theorem \ref{T5.Main1} is proved as follows.

Let $g\in H^{1/2}(\p\O,\S^1)$. By Proposition \ref{P5.BBM1AuxProp}.3, there is $(g_n)_n\subset\mathcal{H}$ s.t. $g_n\to g$ in $H^{1/2}(\p\O)$.


Let $\v\in(0,1)$ and $\delta>0$. Then, by \eqref{5.ComparaisonEnergyProd}, we have
\[
(1-\delta)\frac{f_{\v,g_n}}{|\ln\v|}-C(\delta)\frac{f_{\v,g/g_n}}{|\ln\v|}\leq \frac{f_{\v,g}}{|\ln\v|}\leq (1+\delta)\frac{f_{\v,g_n}}{|\ln\v|}+C(\delta)\frac{f_{\v,g/g_n}}{|\ln\v|}.
\]
From \eqref{5.BorneGrossiereEnergy} and the fact that Theorem  \ref{T5.Main1} holds for $g_n$, we have
\begin{eqnarray}\nonumber
(1-\delta)\pi L(g_n,d_{a^2})-C'(\delta)|g/g_n|_{H^{1/2}}&\leq&\liminf_\v\frac{f_{\v,g}}{|\ln\v|}
\\\nonumber
&\leq&\limsup_\v \frac{f_{\v,g}}{|\ln\v|}
\\\label{5.NumRajouter}
&\leq&(1+\delta)\pi L(g_n,d_{a^2})+C'(\delta)|g/g_n|_{H^{1/2}}.
\end{eqnarray}
Using \eqref{5.ContLongMin}, we obtain that $L(g_n,d_{a^2})\to L(g,d_{a^2})$. If, in \eqref{5.NumRajouter}, we first let $n\to\infty$, we use \eqref{5.ConvQuotient} and we next let $\delta\to0$, we obtain that
\[
\lim_\v\frac{f_{\v,g}}{|\ln\v|}=\pi L(g,d_{a^2}).
\]
The proof of Theorem \ref{T5.Main1} is complete.

\noindent{\bf Acknowledgements.} The author is grateful to Professor Petru Mironescu for having introduced him to the subject of this paper and for stimulating conversations. He wants also to thank Professor Pierre Bousquet for helpful comments. The author gratefully acknowledge the comments of the anonymous Referee.

\appendix



\section{Proof of Proposition \ref{P5.PropertiesdK}}\label{S5.ProofP5.PropertiesdK}

We prove the first assertion. There are two cases to consider: $x\in K$ and $x\notin K$.

If $x\in K$ and $y\notin K$, then we have the existence of a unique point $y_0\in K$ which minimizes $d_{a^2}(y,z)$ among the points $z\in K$. Clearly considering $\Gamma$ a geodesic in $(\R^3,d_{a^2})$ joining $y$ with $y_0$, by definition of $y_0$, $\Gamma\cap K=\emptyset$. Thus $\Gamma$ is a minimal $K$-curve joining $x$ with $y$ according to the definition given above.

If $x,y\notin K$, then we consider $\Gamma$ a geodesic joining $x,y$ in $(\R^3,d_{a^2})$ and, for $z\in\{x,y\}$, let $\Gamma_z$ be a minimal curve in $(\R^3,d_{a^2})$ joining $z$ with $K$ ({\it i.e.} joining $z$ with $x_0$). 

If ${\rm long}_{a^2}(\Gamma_x)+{\rm long}_{a^2}(\Gamma_y)< {\rm long}_{a^2}(\Gamma\setminus K)\leq d_{a^2}(x,y)$, then one may consider $\Gamma_x\cup\Gamma_z$ as a minimal $K$-curve. Indeed, in this situation, $d^K_{a^2}(x,y)< d_{a^2}(x,y)$ which implies that a minimizing sequence of $K$-curves  $\tilde{\Gamma}_n$ satisfies for large $n$ that $\tilde{\Gamma}_n$ contains curves with an endpoint on $\p K$. More precisely, by definition, there are $\Gamma^n_x,\Gamma^n_y$ two connected components of $\tilde{\Gamma}_n$ s.t. for $z\in\{x,y\}$, $\Gamma^n_z$ has $z$ and $z_n'$ for endpoints with $z_n'\in\p K$. Therefore 
\[
{\rm long}_{a^2}(\Gamma_x)+{\rm long}_{a^2}(\Gamma_y)\leq {\rm long}_{a^2}(\tilde{\Gamma}_n),
\]
and thus $\Gamma_x\cup\Gamma_z$ is a minimal $K$-curve.

Otherwise, ${\rm long}_{a^2}(\Gamma_x)+{\rm long}_{a^2}(\Gamma_y)\geq d_{a^2}(x,y)$. Consequently, denoting $\Gamma$ a geodesic in $(\R^3,d_{a^2})$ joining $x$ with $y$, $\Gamma\setminus K$ is a $K$-curve and has a minimal length.

It remains to prove that $\Gamma$, a minimal $K$-curve, is a union of at most two geodesics in $(\R^3,d_{a^2})$. If $\Gamma$ is connected, then, by the definition of a $K$-curve, $\Gamma\cap K=\emptyset$. Thus $\Gamma$ is a geodesic joining $x,y$.

Otherwise, assume that $\Gamma$ is not connected. By the definition of a $K$-curve and by the minimality of $\Gamma$, for $z\in\{x,y\}$, there are $z'\in\p K$ and $\Gamma_z$ a connected component of $\Gamma$ s.t. $z,z'$ are the endpoints of $\Gamma_z$. Thus, by minimality of $\Gamma$, $\Gamma_z$ is a geodesic joining $z,z'$ in $(\R^3,d_{a^2})$ and $\Gamma=\Gamma_x\cup\Gamma_y$.

Now we prove the second assertion. First, we assume that $x_0\notin\p\o$ and that $x_0$ is on a geodesic curve joining $x,y$ in $(\R^3,d_{a^2})$. 

Consider $r_{x_0,x,y}=10^{-2}\min\left\{|x-x_0|,|y-x_0|,\dist(x_0,\p\o)\right\}$. Then, for $r<r_{x_0,x,y}$, considering $K=\overline{B(x_0,r)}$ and a $K$-curve $\Gamma\setminus K$ where $\Gamma$ is a geodesic joining $x,y$ in $(\R^3,d_{a^2})$ and containing $x_0$, we obtain that 
\begin{equation}\label{5.AConditionForKMinimalCurve1}
d^K_{a^2}(x,y)\leq d_{a^2}(x,y)-2a^2(x_0)r.
\end{equation}
This comes from the fact that $\Gamma\cap K$ is a diameter of $K$ and that this diameter is contained in the same connected component of $ \R^3\setminus\p\o$ as $x_0$. To obtain the reverse estimate, it suffices to consider $\Gamma$, a minimal $K$-curve joining $x,y$. From \eqref{5.AConditionForKMinimalCurve1}, we know that $\Gamma$ has exactly two connected components: $\Gamma_x,\Gamma_y$ with $\Gamma_z$ has $z,z'$ for endpoints with $z\in\{x,y\}$ and $z'\in\p K$. Thus it suffices to complete $\Gamma$ by the line segments $[x_0,x']$ and $[x_0,y']$ to obtain the reverse inequality. (Note that in this situation, $[x',y']$ is a diameter of $K$)

If $x_0\in\p\o$, then the argument is similar taking $0<r_{x_0,x,y}<10^{-2}\min\{|x-x_0|,|y-x_0|\}$ sufficiently small s.t.: 
\begin{enumerate}[$\bullet$]
\item $B(x_0,r_{x_0,x,y})\setminus\p\o$ has exactly two connected components, 
\item For all geodesic $\Gamma$ joining $x,y$ in $(\R^3,d_{a^2})$, if $x_0\in\Gamma$ then $(\Gamma\cap K)\setminus\p\o$ has exactly two connected components: one in $\o$ and the other in $\R^3\setminus \overline{\o}$. 
\end{enumerate}
Note that from Proposition \ref{P5.ClassificationGeodesiquespecialcase}, Assertion {\it 4.d.}, $r_{x_0,x,y}$ is well defined.

Now we prove the last assertion arguing by contradiction. Assume that there is $r_n\downarrow0$ s.t. denoting $K_n=\overline{B(x_0,r_n)}$, we have $d^{K_n}_{a^2}(x,y)<d_{a^2}(x,y)$. Consequently there are $x_n,y_n\in\p K_n$ and $\Gamma_n=\Gamma_x^n\cup\Gamma_y^n$ where $\Gamma_z^n$ is a geodesic joining $z$ and $z_n$ in $(\R^3,d_{a^2})$, $z\in\{x,y\}$. Consequently, for $z\in\{x,y\}$, one may complete $\Gamma_z^n$ by the line segment $[z_n',x_0]$ whose length in $(\R^3,d_{a^2})$ is at most $r_n$. We denote this curve by $\tilde{\Gamma}_z^n$. Clearly $d_{a^2}(z,x_0)\leq{\rm long}_{a^2}(\tilde{\Gamma}_z^n)\leq {\rm long}_{a^2}({\Gamma}_z^n) +r_n$. 

It suffices to claim that in a metric space $(X,d)$ which admits geodesic curves we have for $x_0,x,y$ three distinct points in $X$
\[
x_0\text{ is on a geodesic joining }x,y\Longleftrightarrow d(x,y)=d(x,x_0)+d(x_0,y).
\]

Since $x_0$ is not on a geodesic curve joining $x,y$ in $(\R^3,d_{a^2})$, there is $\eta>0$ s.t. $d_{a^2}(x,y)+\eta<d_{a^2}(x,x_0)+d_{a^2}(x_0,y)$ and thus
\[
{\rm long}_{a^2}(\Gamma_x^n)+{\rm long}_{a^2}(\Gamma_y^n)=d_{a^2}^{K_n}(x,y)< d_{a^2}(x,y)\leq {\rm long}_{a^2}(\Gamma_x^n)+{\rm long}_{a^2}(\Gamma_y^n)+2r_n-\eta.
\]
Clearly we obtain a contradiction for $n$ sufficiently large s.t. $r_n<\eta/2$.



\vspace{2mm}


\bibliography{BiblioMain}
\end{document}